\documentclass[10pt,a4paper]{amsart}
\usepackage[utf8]{inputenc}
\usepackage[T1]{fontenc}
\usepackage{amsmath}
\usepackage{upquote}
\usepackage{amsfonts}
\usepackage{amssymb}
\usepackage[english]{babel}
\usepackage{tabularx}
\usepackage{mathtools}
\usepackage{adjustbox}
\usepackage{graphics}
\usepackage{hyperref}
\usepackage{xcolor}
\usepackage{color}
\usepackage{tikz}
\usepackage{chngcntr}
\usepackage{etoolbox}
\usetikzlibrary{calc}
\numberwithin{equation}{section}
\newtheorem{thm}{Theorem}[section]
\newtheorem{pr}[thm]{Proposition}
\newtheorem{lm}[thm]{Lemma}
\newtheorem{re}[thm]{Remark}
\newtheorem{df}[thm]{Definition}
\newtheorem{ex}[thm]{Example}
\newtheorem{cor}[thm]{Corollary}

\newtheorem{con}[thm]{Conjecture}
\newtheoremstyle{case}{}{}{}{}{}{:}{ }{}
\theoremstyle{case}
\newtheorem{case}{Case}
\counterwithin*{case}{thm}
\newtheoremstyle{caso}{}{}{}{}{}{:}{ }{}
\theoremstyle{caso}
\newtheorem{caso}{Case}

\DeclareRobustCommand{\svdots}{
  \vbox{%
    \baselineskip=0.33333\normalbaselineskip
    \lineskiplimit=0pt
    \hbox{.}\hbox{.}\hbox{.}%
    \kern-0.2\baselineskip
  }%
}

\theoremstyle{remark}

\AtEndEnvironment{proof}{\setcounter{claim}{0}}
\newcommand{\lcm}{\text{lcm}}

\newcommand{\R}{\mathbb{R}}

\newcommand{\stirlingone}[2]{\genfrac{[}{]}{0pt}{}{#1}{#2}}

\newcommand{\floor}[1]{\left\lfloor #1 \right\rfloor}

\theoremstyle{remark}
\newcommand\numberthis{\addtocounter{equation}{1}\tag{\theequation}}
\newenvironment{proof*}[1]
  {\renewcommand{\proofname}{#1}%
   \begin{proof}}
  {\end{proof}}
\makeatletter
\let\@@pmod\pmod
\DeclareRobustCommand{\pmod}{\@ifstar\@pmods\@@pmod}
\def\@pmods#1{\mkern4mu({\operator@font mod}\mkern 6mu#1)}
\makeatother

\title[Beyond the log-concavity of the restricted partition function]{Restricted partition functions and the $r$-log-concavity of quasi-polynomial-like functions}
\author{Krystian Gajdzica}
\address{Institute of Mathematics \\
	Faculty of Mathematics and Computer Science \\
	Jagiellonian University in Cracow
}
\email{krystian.gajdzica@doctoral.uj.edu.pl}

\keywords{integer partition; restricted partition function; strong log-concavity; $r$-log-concavity; quasi-polynomial.}

\subjclass[2020]{Primary 11P82, 11P84; Secondary 05A17.}

\begin{document}

\setlength{\parindent}{10mm}
\maketitle

\begin{abstract}
Let $\mathcal{A}=\left(a_i\right)_{i=1}^\infty$ be a weakly increasing sequence of positive integers and let $k$ be a fixed positive integer. For an arbitrary integer $n$, the restricted partition $p_\mathcal{A}(n,k)$ enumerates all the partitions of $n$ whose parts belong to the multiset $\{a_1,a_2,\ldots,a_k\}$. In this paper we investigate some generalizations of the log-concavity of $p_\mathcal{A}(n,k)$. We deal with both some basic extensions like, for instance, the strong log-concavity and a more intriguing challenge that is the $r$-log-concavity of both quasi-polynomial-like functions in general, and the restricted partition function in particular. For each of the problems, we present an efficient solution.
    
\end{abstract}

\section{Introduction}

Let $k$ be a fixed positive integer and let $\mathcal{A}=\left(a_i\right)_{i=1}^\infty$ be a non-decreasing sequence of positive integers (in fact, it is enough to assume that $\mathcal{A}$ is a $k$-tuple).\linebreak A restricted partition $\lambda$ of a non-negative integer $n$ is a sequence of positive integers $\lambda_1,\lambda_2,\ldots,\lambda_j$ such that 
$$n=\lambda_1+\lambda_2+\cdots+\lambda_j,$$
and all of these numbers $\lambda_i$ belong to the multiset $\{a_1,a_2,\ldots,a_k\}$. These elements $\lambda_i$ are further called parts of the partition $\lambda$. Moreover, we also assume that two restricted partitions are considered the same if there is only a difference in the order of their parts. The restricted partition function $p_\mathcal{A}(n,k)$ counts all restricted partitions of $n$. For example, let $\mathcal{A}=(2,3,\textcolor{blue}{3},5,\textcolor{blue}{5},\textcolor{red}{5},7,\textcolor{blue}{7}.\ldots)$ be a sequence of consecutive prime numbers such that the $i$-th prime number appears in $i$ distinct colors. If $k=6$ and $n=7$, then we only take into account those restricted partitions of $7$ whose parts belong to $\{2,3,\textcolor{blue}{3},5,\textcolor{blue}{5},\textcolor{red}{5}\}$, that are: $(\textcolor{red}{5},2)$, $(\textcolor{blue}{5},2)$, $(5,2)$, $(\textcolor{blue}{3},2,2)$, and $(3,2,2)$. Thus, $p_\mathcal{A}(7,6)=5$. We may extend the definition of $p_\mathcal{A}(n,k)$ to all integers simply by setting $p_\mathcal{A}(n,k)=0$ if $n$ is negative. It is also worth noting that $p_\mathcal{A}(0,k)=1$ for every $k\in\mathbb{N}_+$, because of the empty sequence $\lambda=()$. Moreover, if we consider the sequence of consecutive positive integers $\mathcal{A}_1=(1,2,3,4,\ldots)$ and allow `$k=\infty$', then we obtain the well-known partition function $p(n)=p_{\mathcal{A}_1}(n,\infty)$. In $1748$ Euler discovered the generating function for $p(n)$, that is
\begin{align*}
\sum_{n=0}^\infty p(n)x^n=\prod_{i=1}^\infty\frac{1}{1-x^i}.
\end{align*}
For the function $p_\mathcal{A}(n,k)$, the corresponding generating function takes the form
\begin{equation}\label{eq1.1}
    \sum_{n=0}^\infty p_\mathcal{A}(n,k)x^n=\prod_{i=1}^k\frac{1}{1-x^{a_i}}.
\end{equation}

There is an abundance of literature devoted to the restricted partition function. For arithmetic properties of $p_\mathcal{A}(n,k)$, we refer the reader to \cite{KG1, RS, MU}. On the other hand, for some information about the asymptotic behavior we encourage to see \cite{GA, CN, DV}.

The main goal of this manuscript is to generalize some results recently obtained by the author in \cite{KG2}. There is an efficient criterion for the log-concavity of the restricted partition function. Let us recall that a sequence $(w_i)_{i=0}^\infty$ of real numbers is log-concave, if the inequality
\begin{align*}
    w_n^2>w_{n+1}w_{n-1} 
\end{align*}
holds for all $n\geqslant 1$. The main result of the aforementioned paper is the following.

\begin{thm}\label{log3}
Let $\mathcal{A}=\left(a_i\right)_{i=1}^\infty$ be a weakly increasing sequence of positive integers, and let $k\in\mathbb{N}_{\geqslant4}$ be fixed. Suppose further that $\gcd A=1$ for all $(k-2)$-multisubsets $A$ of $\{a_1,a_2,\ldots,a_k\}$. If we have that $k>4$, then both
\begin{align}\label{logi2}
    p_\mathcal{A}^2(n,k)>\left(1+\frac{1}{n^2}\right)p_\mathcal{A}(n+1,k)p_\mathcal{A}(n-1,k)
\end{align}
and
\begin{align}\label{logi1}
    p_\mathcal{A}^2(n,k)>p_\mathcal{A}(n+1,k)p_\mathcal{A}(n-1,k)
\end{align}
hold for all $n\geqslant\frac{2}{k}\prod_{i=1}^k(1+iDk)$, where $D=\lcm{(a_1,a_2,\ldots,a_k)}$. For $k=4$, (\ref{logi1}) remains valid for all such $n$ while (\ref{logi2}) is true for each $n\geqslant\frac{3}{k}\prod_{i=1}^k(1+iDk)$. Additionally, for the constant sequence $\mathcal{A}=(1,\textcolor{blue}{1},\textcolor{red}{1},\ldots)$, we have that (\ref{logi1}) is satisfied for all positive integers $n$ and $k\geqslant2$; and (\ref{logi2}) is true for any integers $k\geqslant3$ and $n\geqslant\frac{k}{k-2}$. 
\end{thm}
The above criterion is optimal in a sense that it can not be extended for any other sequence $\mathcal{A}$ and parameter $k$. That is a consequence of \cite[Corollary 5.8]{KG2}:

\begin{cor}\label{log3cor}
Let $\mathcal{A}=\left(a_i\right)_{i=1}^\infty$ be an arbitrary non-decreasing sequence of positive integers. The sequence $\left(p_\mathcal{A}(n,k)\right)_{n=1}^\infty$ is eventually log-concave if and only if we have $k\geqslant2$ and $a_1=\cdots=a_k=1$ or $k\geqslant4$ and $\gcd A=1$ for all $(k-2)$-multisubsets $A$ of $\{a_1,a_2,\ldots,a_k\}$.
\end{cor}

Now, we investigate some generalizations of Theorem \ref{log3}. Similar to DeSalvo and Pak \cite{DSP} (in the case of $p(n)$), we discuss the so-called \textit{strong log-concavity} of $p_\mathcal{A}(n,k)$, that are inequalities of the form
\begin{align*}
    p_\mathcal{A}^2(n,k)>p_\mathcal{A}(n+m,k)p_\mathcal{A}(n-m,k),
\end{align*}
for all fixed positive integers $m$. The authors prove the following.
\begin{thm}\label{DSP1}
For all $n>m>1$, we have
\begin{align*}
    p^2(n)>p(n+m)p(n-m).
\end{align*}
\end{thm}
It is worth noting that the above statement was firstly conjectured by Chen \cite{Chen} as well as the next one.

\begin{con}\label{DSP2}
For all positive integers $a$ and $b$ such that $a>b$, we have
\begin{align*}
    p^2(an)>p(an+bn)p(an-bn).
\end{align*}
\end{con}
 \noindent Therefore, we also examine under what conditions on a sequence $\mathcal{A}=\left(a_i\right)_{i=1}^\infty$ and $k>0$, the inequality
 \begin{align*}
     p_\mathcal{A}^2(an,k)&>p_\mathcal{A}(an+bn,k)p_\mathcal{A}(an-bn,k)
 \end{align*}
is valid for all sufficiently large values of $n$, where $a$ and $b$ are fixed positive integers such that $a>b$.

Furthermore, DeSalvo and Pak \cite{DSP} resolve Sun's conjecture \cite{Sun}, namely, they prove that the sequence $p(n)/n$ is log-concave for $n\geqslant31.$ On the other hand, we deal with the log-concavity of $p_\mathcal{A}(n,k)/n^\alpha$, where $\alpha$ is an arbitrary positive real number:
\begin{align*}
    \left(\frac{p_\mathcal{A}(n,k)}{n^\alpha}\right)^2&>\left(\frac{p_\mathcal{A}(n+1,k)}{(n+1)^\alpha}\right)\cdot\left(\frac{p_\mathcal{A}(n-1,k)}{(n-1)^\alpha}\right).
\end{align*}

It should be pointed out that research related to log-behavior is not only art for art's sake, but it plays a crucial role in many subjects, and there is a wealth of literature devoted to its applications, see \cite{Asai, Brenti, Doslic2, Doslic3, Huh, Klauder, Milenkovic, Prelec, Shapiro, Stanley2, Warde}. For example, log-concave sequences appear in mathematical biology, where they enumerate secondary structures of some biopolymers \cite{Doslic3, Shapiro}.
Also, the log-concavity plays an important role in signal processing in the so-called white noise theory, see \cite{Asai}. Moreover, Huh \cite{Huh} discovered an interesting interplay between log-concavity and purely chromatic graph theory. On the other hand, log-convex sequences ($w_n^2<w_{n-1}w_{n+1}$) are used in quantum physics for constructing generalized coherent states in models with discrete non-linear spectra, see \cite{Klauder}. In probability theory, Warde and Katti \cite{Warde} showed that there is also a simple sufficient condition for a discrete random variable $X$ to be infinitely divisible which is directly connected with the log-convexity of the sequence $\left(P\{X=n\}\right)_{n=1}^\infty$. 

We also discuss a bit more involved challenge related to the restricted partition function --- that is the $r$-log-concavity of $p_\mathcal{A}(n,k)$. A sequence $\omega=(w_i)_{i=0}^\infty$ is said to be (asymptotically) $r$-log-concave, if there is a positive integer $N$ such that 
\begin{align*}
    \left(\widehat{\mathcal{L}}\omega\right)_i, 
    \left(\widehat{\mathcal{L}}^2\omega\right)_i,\ldots,
    \left(\widehat{\mathcal{L}}^r\omega\right)_i
\end{align*}
are positive for each $i\geqslant N$, where
\begin{align*}
    \widehat{\mathcal{L}}\omega=\left(w_{i+1}^2-w_{i}w_{i+2}\right)_{i=0}^\infty \text{ and } \widehat{\mathcal{L}}^k\omega=\widehat{\mathcal{L}}\left(\widehat{\mathcal{L}}^{k-1}\omega\right)
\end{align*}
for $k\in\{2,3,\ldots,r\}$.
It is clear that $1$-log-concavity of $(w_i)_{i=0}^\infty$ corresponds to $w_{n+1}^2-w_{n}w_{n+2}>0$; the $2$-log-concavity additionally requires that the inequality
$$w_{n+2}^4-2w_{n+1}w_{n+2}^2w_{n+3}+w_{n+1}^2w_{n+2}w_{n+4}+w_{n}w_{n+2}w_{n+3}^2-w_{n}w_{n+2}^2w_{n+4}>0$$
is satisfied for all $n\geqslant N$. Clearly, as $r$ grows, the required conditions become more and more intricate. 

Despite the fact that the $r$-log-concavity problem for a given sequence is\linebreak a highly non-trivial task in general, Hou and Zhang \cite{Hou} discovered an effective criterion for its solution (see, Theorem \ref{HouZ} below). Moreover, these authors used that result and proved that the partition function $p(n)$ is (asymptotically) $r$-log-concave for each $r\in\mathbb{N}_+$, for more details see \cite{Hou2}. Afterwards, Mukherjee, Zhang and Zhong \cite{Mu} applied the aforementioned criterion and showed that the overpartition function is (asymptotically) $r$-log-concave for any $r\in\mathbb{N}_+$. Recall that an overpartition \cite{Corteel} of an integer $n$ is a partition of $n$ where the first occurrence of every distinct part may be overlined. The overpartition function $\overline{p}(n)$ denotes the number of all overpartitions of $n$. There is a rich and vast literature dealing with the analogous problems for other sequences, we encourage the reader to see, for instance, \cite{Butler, Chen1, Dawsey, Guo1, Heim1, Heim2, Krattenthaler, Leroux, Ono, Sagan, Sagan1, MU2, Wagner}.

We present a similar result to those mentioned above in the case of the restricted partition function $p_\mathcal{A}(n,k)$. In fact, we show a bit more general criterion for $r$-log-concavity of quasi-polynomial-like functions. By a quasi-polynomial-like function we mean any expression of the form
\begin{align*}
    f(n)=a_l(n)n^l+a_{l-1}(n)n^{l-1}+\cdots+a_{l-s}(n)n^{l-s}+o\left(n^{l-s}\right),
\end{align*}
where $l$ and $s$ are fixed positive integers such that $l\geqslant s$, $a_{l-s}(n),\ldots,a_l(n)$ are real coefficients depending on the residue class of $n\pmod*{M}$ for some $M\geqslant2$, and $o$ denotes the standard little-$o$ notation. The main result of this paper presents, in some sense, optimal requirements on the function $f(n)$ to prove that it is $r$-log-concave for any $r\geqslant1$.

This manuscript is organized as follows. Section 2 contains necessary notations, properties and tools that are used throughout the article. Section 3 is devoted to some basic generalizations of the log-concavity criterion for the restricted partition function. We investigate the analogues of both Theorem \ref{DSP1} and Conjecture \ref{DSP2}, as well as the log-concavity of the sequence $\left(p_\mathcal{A}(n,k)/n^\alpha\right)_{n=1}^\infty$. Finally, Section 4 studies the $r$-log-concavity of both quasi-polynomial-like functions and the restricted partition function $p_\mathcal{A}(n,k)$. To deal with these problems, we use methods from Hou and Zhang's paper \cite{Hou}. Actually, we also present an alternative more direct approach which is effective in the case of quasi-polynomial-like functions.

\section{Preliminaries}

At the beginning, let us fix some notations. The set of non-negative integers is denoted by $\mathbb{N}$, while $\mathbb{N}_+$ and $\mathbb{R}_+$ correspond to the sets of positive integers and positive real numbers, respectively. Additionally, we put $\mathbb{N}_{\geqslant k}=\mathbb{N}\setminus \{0,1,\ldots,k-1\}$.

Henceforth, $\mathcal{A}=(a_i)_{i=1}^\infty$ denotes a weakly increasing sequence of positive integers. Let $k$ and $n$ be fixed positive integers. The restricted partition function $p_\mathcal{A}(n,k)$ enumerates all partitions of $n$ with parts in the multiset $\{a_1,a_2,\ldots,a_k\}$. We additionally set $p_\mathcal{A}(0,k)=1$ and $p_\mathcal{A}(n,k)=0$ if $n<0$. From the equality (\ref{eq1.1}), one can easily derive a recurrence relation for the restricted partition function of the form
\begin{equation*}
p_\mathcal{A}(n,k)=p_\mathcal{A}(n-a_k,k)+p_\mathcal{A}(n,k-1),
\end{equation*}
where $k\geqslant2$. For $k=1$, we have
\begin{align*}
	        p_\mathcal{A}(n,1)=\begin{cases}
        0, & \text{if } a_1\nmid n,\\
        1, & \text{if } a_1\mid n.
        \end{cases}
	    \end{align*}
It turns out that restricted partition function $p_\mathcal{A}(n,k)$ behaves `almost like' a polynomial of degree $k-1$. More precisely, it is a $quasi$-polynomial of the form
\begin{align*}
    p_\mathcal{A}(n,k)=c_{k-1}(r)n^{k-1}+c_{k-2}(r)n^{k-2}+\cdots+c_0(r),
\end{align*}
where every $c_j(r)$ depends on the residue class $r$ of $n\pmod*{\lcm(a_1,a_2,\ldots,a_k)}$ for $0\leqslant j \leqslant k-1$ (for additional information about quasi-polynomials, we refer the reader to Stanley's book \cite[Section~4.4]{Stanley}). This result goes back to Bell \cite{Bell}, who showed it via partial fraction decomposition of the related rational generating function. However, there are more concrete estimates for the restricted partition function. For instance, Almkvist \cite{GA} obtained a very elegant result related to the asymptotic behavior of $p_\mathcal{A}(n,k)$. Let us define symmetric polynomials $\sigma_i(x_1,x_2,\ldots,x_k)$ by the power series expansion
\begin{align*}
\prod_{i=1}^k\frac{x_it/2}{\sinh(x_it/2)}=\sum_{m=0}^\infty\sigma_m(x_1,x_2,\ldots,x_k)t^m.
\end{align*} 
Then Almkvist's characterization of $p_\mathcal{A}(n,k)$ is the following.
\begin{thm}\label{2.4}
Let $\mathcal{A}=\left(a_i\right)_{i=1}^\infty$ and $k\in\mathbb{N}_{\geqslant2}$ be fixed. For a given integer $j\in\{1,2,\ldots,k\}$, if $\gcd A=1$ for every $j$-element multisubset ($j$-multisubset) $A$ of $\{a_1,a_2,\ldots,a_k\}$ and $\sigma=a_1+a_2+\cdots+a_k$, then
\begin{equation*}
p_\mathcal{A}(n,k)=\frac{1}{\prod_{i=1}^ka_i}\sum_{i=0}^{k-j}\sigma_i(a_1,a_2,\ldots,a_k)\frac{(n+\sigma/2)^{k-1-i}}{(k-1-i)!}+O(n^{j-2})
\end{equation*}
as $n\to\infty$.
\end{thm}
One can verify that $\sigma_0=1$, and $\sigma_i=0$ if $i$ is odd. Moreover, if we put $s_i=a_1^i+a_2^i+\cdots+a_k^i$, then
\begin{align*}
\sigma_2=-\frac{s_2}{24}\text{,}\hspace{0.2cm}\sigma_4=\frac{5s_2^2+2s_4}{5760}\text{,}\hspace{0.2cm}\sigma_6=-\frac{35s_2^3+42s_2s_4+16s_6}{2903040}.
\end{align*}
Further, it turns out that converse implication in Theorem \ref{2.4} is also true, because of the following. 

\begin{pr}\label{pr5.7}
Let $\mathcal{A}=\left(a_i\right)_{i=1}^\infty$ and $k\in\mathbb{N}_{+}$ be fixed. If
\begin{equation*}
p_\mathcal{A}(n,k)=c_{k-1}n^{k-1}+c_{k-2}n^{k-2}+\cdots+c_{j-1}n^{j-1}+c_{j-2}(n)n^{j-2}+\cdots+c_0(n),
\end{equation*}
where $c_{k-1},\ldots,c_{j-1}$ are independent of a residue class $n\pmod*{\lcm(a_1,\ldots,a_k)}$, then $\gcd A=1$ for all $j$-multisubsets $A$ of $\{a_1,a_2,\ldots,a_k\}$.
\end{pr}
For more details about the aforementioned result, we refer the reader to \cite{KG2}. It is also worth noting that we have an explicit formula for $p_\mathcal{A}(n,k)$ --- that is a result due to Cimpoeaş and Nicolae \cite{CN} with a small modification. Before we present their theorem, let us recall that the unsigned Stirling number of the first kind $\stirlingone{n}{m}$ is the number of permutations of $n$ elements with exactly $m$ cycles. It might be defined equivalently as the coefficient of the rising factorial (Pochhammer function), namely
\begin{align*}
    x^{\overline{n}}:=x(x+1)\cdots(x+n-1)=\sum_{i=0}^{n}\stirlingone{n}{i}x^i.
\end{align*}
We encourage the reader to see \cite[Chapter~6]{GKP} for more information about Stirling numbers.

\begin{thm}\label{CNT}
For any $\mathcal{A}=\left(a_i\right)_{i=1}^\infty$ and $k\geqslant1$, we have
\begin{align*}
    p_\mathcal{A}(n,k)=\frac{1}{(k-1)!}\sum_{m=0}^{k-1}\sum_{\substack{ 0\leqslant j_1\leqslant\frac{D}{a_1}-1\\ \svdots \\ 0\leqslant j_k\leqslant\frac{D}{a_k}-1\\
    S\equiv n\pmod*{D}}}\sum_{i=m}^{k-1}\stirlingone{k}{i+1}(-1)^{i-m}\binom{i}{m}D^{-i}S^{i-m}n^m,
\end{align*}
where $D=\lcm(a_1,a_2,\ldots,a_k)$ and $S=a_1j_1+\cdots+a_kj_k$.
\end{thm}

\begin{re}\label{remark1}
{\rm One can notice that there is a difference between the statement in Theorem \ref{CNT} and the original formula from Cimpoeaş and Nicolae's paper \cite[Theorem 2.8]{CN} which states that
\begin{align*}
    p_\mathcal{A}(n,k)=\frac{1}{(k-1)!}\sum_{m=0}^{n-1}\sum_{\substack{ 0\leqslant j_1\leqslant\frac{D}{a_1}-1\\ \svdots \\ 0\leqslant j_k\leqslant\frac{D}{a_k}-1\\
    S\equiv n\pmod*{D}}}\sum_{i=m}^{k-1}\stirlingone{k}{i}(-1)^{i-m}\binom{i}{m}D^{-i}S^{i-m}n^m
\end{align*}
with the same notation as before. At first, there is a typographical error over the first sum, namely, there has to be $k-1$ instead of $n-1$. Furthermore, we need to replace $\stirlingone{k}{i}$ by $\stirlingone{k}{i+1}$ --- that is a consequence of the equality 
\begin{align*}
    \binom{n+k-1}{k-1}=\frac{1}{(k-1)!}\sum_{i=0}^{k-1}\stirlingone{k}{i+1}n^i,
\end{align*}
which is accidentally misspelled at the beginning of Section 2 in \cite{CN}.

For the sake of clarity, it should be also pointed out that in Theorem \ref{CNT}, we may have that $S=0$. If this is the case, we assume that $0^0=1$. 
}
\end{re}
From the above-mentioned properties, the author \cite{KG2} deduced the following lower and upper bounds for the restricted partition function. 

\begin{pr}\label{Proposition4}
Let $\mathcal{A}=\left(a_i\right)_{i=1}^\infty$ and $k\in\mathbb{N}_{\geqslant2}$ be fixed. For a given integer $j\in\{2,3,\ldots,k\}$, if $\gcd A=1$ for all $j$-multisubsets $A$ of $\{a_1,a_2,\ldots,a_k\}$, then
\begin{equation*}
c_{k-1}n^{k-1}+\cdots+c_{j-1}n^{j-1}-Fn^{j-2}<p_\mathcal{A}(n,k)<c_{k-1}n^{k-1}+\cdots+c_{j-1}n^{j-1}+Fn^{j-2}
\end{equation*}
holds for every $n>0$, where all the coefficients $c_i$ are uniquely determined by Theorem \ref{2.4}, and $F=\frac{\prod_{i=1}^k(1+iDk)}{k!\prod_{i=1}^ka_i}$ with $D=\lcm(a_1,a_2,\ldots,a_k)$.
\end{pr}

It is worth mentioning that the constant $F$ appearing in Proposition \ref{Proposition4} is not optimal, nevertheless we will soon observe the usefulness of the result.

\section{Some basic extensions of the log-concavity criterion for \texorpdfstring{$p_\mathcal{A}(n,k)$}{TEXT}}

Our first goal is to resolve a few of the problems presented in Introduction, that are 
\begin{align}
    p_\mathcal{A}^2(n,k)&>p_\mathcal{A}(n+m,k)p_\mathcal{A}(n-m,k),\label{problem1}\\
    p_\mathcal{A}^2(an,k)&>p_\mathcal{A}(an+bn,k)p_\mathcal{A}(an-bn,k),\label{problem2}\\
    \left(\frac{p_\mathcal{A}(n,k)}{n^\alpha}\right)^2&>\left(\frac{p_\mathcal{A}(n+1,k)}{(n+1)^\alpha}\right)\cdot\left(\frac{p_\mathcal{A}(n-1,k)}{(n-1)^\alpha}\right)\label{problem3},
\end{align}
where $a,b,m$ are fixed positive integers such that $a>b$ while $\alpha$ is an arbitrary positive real number. Essentially, all of them are very similar to Theorem \ref{log3} in the proof. Let us present a heuristic reasoning for the inequality (\ref{problem1}). The idea is quite easy, we apply Proposition \ref{Proposition4} in order to bound both $p_\mathcal{A}^2(n,k)$ from below and $p_\mathcal{A}(n+m,k)p_\mathcal{A}(n-m,k)$ from above. Subsequently, we just compare these bounds and examine when the relevant inequality between them holds.

\begin{thm}\label{thm1}
    Let $\mathcal{A}=\left(a_i\right)_{i=1}^\infty$ and $m\in\mathbb{N}_{\geqslant2}$ be fixed. Suppose farther that $k\geqslant4$ and $\gcd A=1$ for every $(k-2)$-multisubset $A\subset\{a_1,a_2,\ldots,a_k\}$. If $2\hat{d}\geqslant3m^3$, then the inequality
    \begin{align*}
        p_\mathcal{A}^2(n,k)>p_\mathcal{A}(n+m,k)p_\mathcal{A}(n-m,k)
    \end{align*}
    holds for all $n\geqslant\hat{d}$, where $\hat{d}=k^{-1}\prod_{i=1}^k(1+iDk)$ with $D=\lcm(a_1,a_2,\ldots,a_k)$; otherwise it is valid for any $n\geqslant3m$. Additionally, for the constant sequence $\mathcal{A}=(1,\textcolor{blue}{1},\textcolor{red}{1},\ldots)$, the above inequality is satisfied for all positive integers $n$ and $k\geqslant2$. 
\end{thm}
\begin{proof}
First, let $\mathcal{A}=(1,\textcolor{blue}{1},\textcolor{red}{1},\ldots)$ be a constant sequence. Observe that Theorem \ref{CNT} maintains that
\begin{align}\label{formula11}
            p_\mathcal{A}(n,k)=\binom{n+k-1}{k-1}.
\end{align}
Thus, the second part of the statement can be directly verified. For the first one, let $k$, $m$ and $\{a_1,a_2,\ldots,a_k\}$ satisfy the assumptions in the theorem. It follows from Proposition \ref{Proposition4} that the inequalities
\begin{align*}
    an^{k-1}+bn^{k-2}+cn^{k-3}-dn^{k-4}<p_\mathcal{A}(n,k)<an^{k-1}+bn^{k-2}+cn^{k-3}+dn^{k-4}
\end{align*}
hold for all $n$, where the coefficients $a, b, c$ and $d$ are given by
\begin{align*}
    a&=\frac{1}{(k-1)!\prod_{i=1}^ka_i},\\
    b&=\frac{\sigma}{2(k-2)!\prod_{i=1}^ka_i},\\
    c&=\frac{3\sigma^2-s_2}{24(k-3)!\prod_{i=1}^ka_i},\\
    d&=\frac{\prod_{i=1}^k(1+iDk)}{k!\prod_{i=1}^ka_i}
\end{align*}
with $\sigma=a_1+a_2+\cdots+a_k$, $s_2=a_1^2+a_2^2+\cdots+a_k^2$ and $D=\lcm(a_1,a_2,\ldots,a_k)$. Now, it is enough to check when the following inequality
\begin{align*}
    &\left[an^{k-1}+bn^{k-2}+cn^{k-3}-dn^{k-4}\right]^2\\
    >&\left[a(n+m)^{k-1}+b(n+m)^{k-2}+c(n+m)^{k-3}+d(n+m)^{k-4}\right]\\
    &\times\left[a(n-m)^{k-1}+b(n-m)^{k-2}+c(n-m)^{k-3}+d(n-m)^{k-4}\right]
\end{align*}
is true. The above might be simplified to
\begin{align*}  
 3m^2n^4+4(\hat{b}m^2-\hat{d})n^3&+(2\hat{b}^2m^2-3m^4-4\hat{b}\hat{d})n^2+2(\hat{b}\hat{c}m^2-\hat{b}m^4-3\hat{d}m^2-2\hat{c}\hat{d})n\\
 &+\hat{c}^2m^2-2\hat{b}\hat{d}m^2-\hat{b}^2m^4+2\hat{c}m^4+m^6>0,
\end{align*}
where $\hat{b}=b/a$, $\hat{c}=c/a$ and $\hat{d}=d/a$. Next, let $f(n)$ denote the polynomial on the left hand side. Both real roots of the second derivative of $f(n)$ are given by
\begin{align*}
    n_1&=\frac{2\hat{d}-2\hat{b}m^2-\sqrt{2}\sqrt{2\hat{d}^2+3m^6}}{6m^2},\\
     n_2&=\frac{2\hat{d}-2\hat{b}m^2+\sqrt{2}\sqrt{2\hat{d}^2+3m^6}}{6m^2}.
\end{align*}
Now, it is convenient to investigate two separately cases. Either we have that $2\hat{d}^2\geqslant3m^6$ or $2\hat{d}^2<3m^6$.
\begin{case}
    If the inequality $2\hat{d}^2\geqslant3m^6$ holds, then $n_2\leqslant\hat{d}$. But, this implies that $f'(n)$ is increasing for all $n\geqslant\hat{d}$. One can also easy verify that $f'(\hat{d})>0$ and $f(\hat{d})>0$, so we get that both $f'(n)$ and $f(n)$ are positive for every $n\geqslant\hat{d}$.
\end{case}
\begin{case}
    Since we assume that $2\hat{d}^2<3m^6$, it is straightforward to see that $n_2\leqslant m$. Thus $f'(n)$ grows for $n\geqslant m$. Furthermore, one may check that $f'(3m)>0$ as well as $f(3m)>0$. Hence, $f(n)$ is positive for all $n\geqslant 3m$, which finishes the proof.
\end{case}
\end{proof}
Let us note that the case of $m=1$ is omitted in the above, because it is contained in Theorem \ref{log3}. It is also worth pointing out that the demands for parameter $n$ in Theorem \ref{thm1} are not optimal. Nevertheless, the result may be applied for a wide class of integer sequences. Furthermore, by virtue of Corollary \ref{log3cor} and Theorem \ref{thm1}, we obtain the following characterization.

\begin{cor}\label{cor1}
The sequence $\left(p_\mathcal{A}(n,k)\right)_{n=1}^\infty$ is eventually strong log-concave (the inequality (\ref{problem1}) is satisfied for any positive integer $m$ and sufficiently large values of $n$) if and only if we have $k\geqslant2$ and $a_1=\cdots=a_k=1$ or $k\geqslant4$ and $\gcd A=1$ for all $(k-2)$-multisubsets $A$ of $\{a_1,a_2,\ldots,a_k\}$.
\end{cor}

\begin{proof}
    It follows directly from Corollary \ref{log3cor} and Theorem \ref{thm1}.
\end{proof}

Now, it is time to deal with the next task of this section. Surprisingly, the following is true.

\begin{cor}\label{cor3}
    Let $\mathcal{A}=\left(a_i\right)_{i=1}^\infty$ and $u,v\in\mathbb{R}_+$ be fixed such that $u>v$. If $k\geqslant2$ and $\gcd(a_1,a_2,\ldots,a_k)=1$, then the inequality
    \begin{align*}
        p_\mathcal{A}^2(un,k)>p_\mathcal{A}(un+vn,k)p_\mathcal{A}(un-vn,k)
    \end{align*}
    is true for every $n>\frac{4u}{kv^2}\prod_{i=1}^k(1+iDk)$ with $D=\lcm(a_1,a_2,\ldots,a_k)$ whenever $un, un+vn$ and $un-vn$ are integers.
\end{cor}
Since the proof of Corollary \ref{cor3} is very easy and analogous to that one of Theorem \ref{thm1}, we leave it as an exercise for the reader. It is worth pointing out that even though the inequalities (\ref{problem1}) and (\ref{problem2}) look very similar, the criteria for their solutions are significantly different. In addition, even if we assume that $\gcd(a_1,a_2,\ldots,a_k)>1$, then (\ref{problem2}) might be still valid for all sufficiently large values of $n$, as it is shown in the following instance.

\begin{ex}
{\rm Let $\mathcal{A}=(2,\textcolor{blue}{2},\textcolor{red}{2},\ldots)$ be a constant sequence and let $k\in\mathbb{N}_+$ be fixed. Suppose further that $u=4$ and $v=2$. Thus we deal with the inequality of the form
\begin{align}\label{ex1}
p_\mathcal{A}^2(4n,k)>p_\mathcal{A}(6n,k)p_\mathcal{A}(2n,k).
\end{align}
Now, observe that the equality $p_\mathcal{A}(2n,k)=p_\mathcal{B}(n,k)$ holds for every integer $n$, where $\mathcal{B}=(1,\textcolor{blue}{1},\textcolor{red}{1},\ldots)$. Therefore, as a consequence of the equality (\ref{formula11}), we may reduce the problem to
\begin{align*}
    \binom{2n+k-1}{k-1}^2>\binom{3n+k-1}{k-1}\binom{n+k-1}{k-1}.
\end{align*}
Finally, it might be easily checked that the leading coefficient on the left hand side is $4^{k-1}/\left((k-1)!\right)^2$, while on the right hand side it is $3^{k-1}/\left((k-1)!\right)^2$, and both of them stand next to $n^{2k-2}$. Hence, for each $k>1$, we get that the inequality (\ref{ex1}) is satisfied for all but finitely many positive integers $n$.}
\end{ex}

The above example shows that the inequality (\ref{problem2}) is not a really interesting problem for the restricted partition function, despite the fact that its analogue for $p(n)$ is (see, \cite{Chen, DSP}).

In the last part of this section, we present a complete solution of the inequality (\ref{problem3}). It turns out to be a more fascinating challenge than the previous ones, as the following result suggests.

\begin{thm}\label{thm2}
    Let $\mathcal{A}=\left(a_i\right)_{i=1}^\infty$ and $\alpha\in\mathbb{R}_+$ be fixed. The sequence $\left(\frac{p_\mathcal{A}(n,k)}{n^\alpha}\right)_{n=1}^\infty$ is eventually log-concave if and only if we have $k>\alpha+1$ and either $2\leqslant k\leqslant3$ and $a_1=a_2=a_3=1$ or $k\geqslant4$ and $\gcd A=1$ for all $(k-2)$-multisubsets $A$ of $\{a_1,a_2,\ldots,a_k\}$.
    
\end{thm}

\begin{proof} 
    Let $\mathcal{A}=\left(a_i\right)_{i=1}^\infty$ and $\alpha$ be as in the statement. At the beginning, we prove the implication to the left hand side. We just need to solve the inequality 
    \begin{align*}
        (n^2-1)^\alpha p_\mathcal{A}^2(n,k)-n^{2\alpha}p_\mathcal{A}(n+1,k)p_\mathcal{A}(n-1,k)>0.
    \end{align*}
    Our first goal is to bound the left hand side from below. In order to do that we apply the generalized binomial theorem which maintains that
    \begin{align*}
        (n^2-1)^\alpha=\sum_{j=0}^\infty(-1)^j\binom{\alpha}{j}n^{2(\alpha-j)},
    \end{align*}
    where 
    \begin{align*}
        \binom{\alpha}{j}=\frac{\alpha(\alpha-1)\cdots(\alpha-j+1)}{j!}.
    \end{align*}
    Since we must have $n\geqslant2$, it is clear that the following inequalities
    \begin{align*}
        (n^2-1)^\alpha&>\sum_{j=0}^{\floor{\alpha}+1}(-1)^j\binom{\alpha}{j}n^{2(\alpha-j)}-\binom{\alpha}{\floor{\alpha}}n^{2(\alpha-\floor{\alpha}-2)}\sum_{j=0}^\infty n^{-2j}\\
        &>\sum_{j=0}^{\floor{\alpha}+1}(-1)^j\binom{\alpha}{j}n^{2(\alpha-j)}-2\binom{\alpha}{\floor{\alpha}}n^{2(\alpha-\floor{\alpha}-2)}=f(n)
    \end{align*}
    are valid. We consider two cases depending on both the value of $k>\alpha+1$ and the multiset $\{a_1,a_2,\ldots,a_k\}$.
    \begin{case}\label{thm2case1}
        First, we assume that $k\geqslant4$ and $\gcd A=1$ for all $(k-2)$-multisubsets $A$ of $\{a_1,a_2,\ldots,a_k\}$. It is enough to deal with 
    \begin{align*}
 f(n)p_\mathcal{A}^2(n,k)-n^{2\alpha}p_\mathcal{A}(n+1,k)p_\mathcal{A}(n-1,k)>0.
    \end{align*}
    From the assumptions on the multiset $\{a_1,\ldots,a_k\}$ and Theorem \ref{2.4}, we get that 
    \begin{align*}
f(n)p_\mathcal{A}^2(n,k)-n^{2\alpha}p_\mathcal{A}(n+1,k)p_\mathcal{A}(n-1,k)=(k-1-\alpha)a^2n^{2k+2\alpha-4}+\widetilde{f}(n), 
    \end{align*}
   where $$a=\frac{1}{(k-1)!\prod_{i=1}^k{a_i}}$$
    and $\widetilde{f}(n)$ is a generalized quasi-polynomial of degree $2k+2\alpha-5$, that is an expression of the form 
    $$\widetilde{f}(n)=n^{2\left(\alpha-\floor{\alpha}-2\right)}\sum_{j=0}^{2k+2\floor{\alpha}-1}t_j(n)n^j,$$
    where each $t_m(n)$ is some real number depending on the residue class of \linebreak $n\pmod*{\lcm(a_1,a_2,\ldots,a_k)}$. Since $k>\alpha+1$, the leading coefficient $(k-1-\alpha)a^2$ is positive. Therefore, we obtain that the sequence $p_\mathcal{A}(n,k)/n^\alpha$ is log-concave for all but finitely many positive integers $n$.
    \end{case}
    \begin{case}
        Now, let us assume that $2\leqslant k\leqslant3$ and $a_1=a_2=a_3=1$. These two alternatives might be investigated separately by repeating the reasoning from Case \ref{thm2case1} and employing the formula (\ref{formula11}). We encourage the reader to check all the details on their own.
    \end{case}

    To prove the implication to the right hand side let us fix $k\in\mathbb{N}_+$. It is clear that the sequence $\left(p_\mathcal{A}(n,1)/n^\alpha\right)_{n=1}^\infty$ can not be log-concave. Therefore, let $k\geqslant2$ and suppose, for contradiction, that the assumptions on the numbers $a_1,a_2,\ldots,a_k$ and $k$ do not hold. Before we go into the main part of the proof, let us observe that we may also determine an upper bound for $(n^2-1)^\alpha$. Analogously to the computations from the beginning, one can easily get that
    \begin{align*}
        (n^2-1)^\alpha<\sum_{j=0}^{\floor{\alpha}+1}(-1)^j\binom{\alpha}{j}n^{2(\alpha-j)}+2\binom{\alpha}{\floor{\alpha}}n^{2(\alpha-\floor{\alpha}-2)}=g(n).
    \end{align*}
    Once again we have a few possibilities to examine. In order to make the text more transparent, we label each of the cases by its general assumptions.
    \begin{caso}[$4\leqslant k<\alpha+1$]\label{caso1}
    At first, we also assume that $\gcd A=1$ for all $(k-2)$-multisubsets $A$ of $\{a_1,a_2,\ldots,a_k\}$. In order to show that $p_\mathcal{A}(n,k)/n^\alpha$ can not be log-concave, it suffices to note that the leading coefficient of
    \begin{align*}
        g(n)p_\mathcal{A}^2(n,k)-n^{2\alpha}p_\mathcal{A}(n+1,k)p_\mathcal{A}(n-1,k)=(k-1-\alpha)a^2n^{2k+2\alpha-4}+\widetilde{g}(n),
    \end{align*}
    where $a$ is as before and $\widetilde{g}(n)$ is some generalized quasi-polynomial of degree $2k-2\alpha-5$, is negative.

    On the other hand, if the assumptions on the multiset $\{a_1,a_2,\ldots,a_k\}$ do not hold, then Proposition \ref{pr5.7} asserts that 
\begin{align*}
    p_\mathcal{A}(n,k)=c_{k-1}(n)n^{k-1}+c_{k-2}(n)n^{k-2}+c_{k-3}(n)n^{k-3}+\cdots+c_0(n),
\end{align*}
where at least one of $c_{k-1}(n)$, $c_{k-2}(n)$, or $c_{k-3}(n)$ depends on a residue class of $n\pmod*{\lcm(a_1,\ldots,a_k)}$. Let $t\in\{k-3,k-2,k-1\}$ be the largest index with this property. Now, it is enough to take any $n\pmod*{\lcm(a_1,\ldots,a_k)}$ such that $c_t(n)$ is the smallest, and at least one of $c_t(n+1)$ or $c_t(n-1)$ is strictly larger than $c_t(n)$. If we do so, then the leading coefficient of the generalized quasi-polynomial
\begin{align*}
    g(n)p_\mathcal{A}^2(n,k)-n^{2\alpha}p_\mathcal{A}(n+1,k)p_\mathcal{A}(n-1,k)
\end{align*}
is negative, and, in particular, the sequence $\left(p_\mathcal{A}(n,k)/n^\alpha\right)_{n=1}^\infty$ can not be log-concave for infinitely many values of $n$.
    \end{caso}
    \begin{caso}[$4\leqslant k=\alpha+1$]
        In contrast to Case \ref{caso1}, if we assume that $\gcd A=1$ for all $(k-2)$-multisubsets $A$ of $\{a_1,a_2,\ldots,a_k\}$, then we obtain that
        \begin{align*}
        g(n)p_\mathcal{A}^2(n,k)-n^{2\alpha}p_\mathcal{A}(n+1,k)p_\mathcal{A}(n-1,k)=\widetilde{g}(n),
    \end{align*}
    where $\widetilde{g}(n)$ is as before. Hence, we need to say something about the leading coefficient of $\widetilde{g}(n)$. There are two possibilities: either the coefficient $c_{k-4}(n)$ of 
    \begin{align*}
    p_\mathcal{A}(n,k)=c_{k-1}(n)n^{k-1}+c_{k-2}(n)n^{k-2}+c_{k-3}(n)n^{k-3}+\cdots+c_0(n)
\end{align*}
is independent of the residue class of $n\pmod*{\lcm(a_1,\ldots,a_k)}$ or not. If $c_{k-4}(n)$ does not depend on the residue class of $n\pmod*{\lcm (a_1,\ldots,a_k)}$, then Proposition \ref{pr5.7} points out that $\gcd A=1$ for any $(k-3)$-multisubset $A$ of $\{a_1,a_2,\ldots,a_k\}$, and we get that
\begin{align*}
    \widetilde{g}(n)=-2abn^{2k+2\alpha-5}+O(n^{2k+2\alpha-6}),
\end{align*}
where $a$ is as before and 
$$b=\frac{a_1+a_2+\cdots+a_k}{2(k-2)!\prod_{i=1}^ka_i}.$$
Thus, the leading coefficient of $\widetilde{g}(n)$ is negative, as required. Let us note here that the case of $k=4$ is also included above. In that situation we must have $a_1=a_2=a_3=a_4=1$ ($\gcd(m)=m$ for any $m\in\mathbb{N}_+$), and $p_\mathcal{A}(n,k)$ is given by (\ref{formula11}).

On the other hand, if at least one of the coefficients $c_{i}(n)$ for $i\in\{k-4,k-3,k-2,k-1\}$ depends on the residue class of $n\pmod*{\lcm(a_1,\ldots,a_k)}$, then we present analogous reasoning to that one from the second part of Case \ref{caso1}, and obtain that the sequence $\left(p_\mathcal{A}(n,k)/n^\alpha\right)_{n=1}^\infty$ can not be log-concave for infinitely many values of $n$. 
    \end{caso}
    \begin{caso}[$2\leqslant k\leqslant3 \text{ }\&\text{ } k\leqslant\alpha+1$]
    Applying preceding methods, it is not difficult to deduce the required result. We leave this case as an exercise for the reader.
    \end{caso}
\end{proof}

At the end of this section, we illustrate how the above criterion works in practice. But, before we do so let us introduce an additional notation which allows us to make the text more clear.
\begin{df}
For given parameters $\mathcal{A}=\left(a_i\right)_{i=1}^\infty$, $k\geqslant1$ and $\alpha\in\mathbb{R}_+$, we set
\begin{align*}
    \Delta_{\mathcal{A},k}^\alpha(n)=(n^2-1)^{\alpha}p_\mathcal{A}^2(n,k)-n^{2\alpha}p_\mathcal{A}(n+1,k)p_\mathcal{A}(n-1,k).
\end{align*}
\end{df}

It is clear that the sequence $\left(p_\mathcal{A}(n,k)/n^\alpha\right)_{n=1}^\infty$ is log-concave if and only if the corresponding function $\Delta_{\mathcal{A},k}^\alpha(n)$ is positive for every $n\geqslant1$. 

\begin{ex}\label{Example2}
{\rm Let $\mathcal{A}_1=(1,2,3,4,5,6,\ldots)$ be a sequence of consecutive positive integers, and let $\alpha_1=4.99$. Computations carried out in Wolfram Mathematica \cite{WM} exhibit us the behavior of $\Delta_{\mathcal{A}_1,k}^{\alpha_1}(n)$ for $4\leqslant k\leqslant7$.
\begin{center}
\begin{figure}[!htb]
   \begin{minipage}{0.47\textwidth}
     \centering
     \includegraphics[width=.9\linewidth]{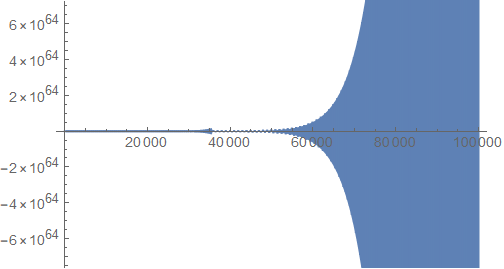}
     \caption{Values of $\Delta_{\mathcal{A}_1,4}^{\alpha_1}(n)$ for $2\leqslant n \leqslant10^5$}
   \end{minipage}\hfill
   \begin{minipage}{0.47\textwidth}
     \centering
     \includegraphics[width=.9\linewidth]{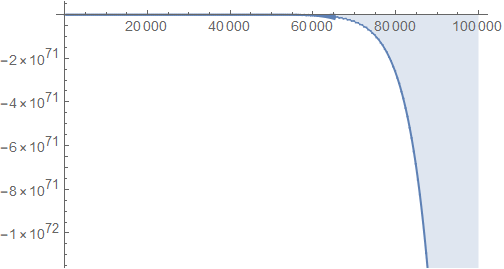}
     \caption{Values of $\Delta_{\mathcal{A}_1,5}^{\alpha_1}(n)$ for $2\leqslant n \leqslant10^5$}
   \end{minipage}
\end{figure}
\begin{figure}[!htb]
   \begin{minipage}{0.47\textwidth}
     \centering
     \includegraphics[width=.9\linewidth]{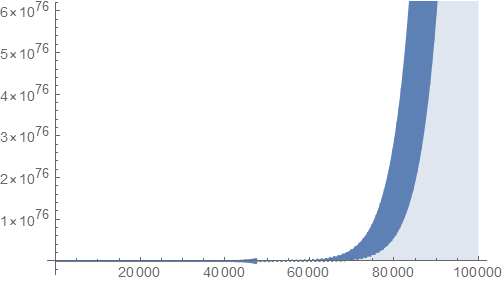}
     \caption{Values of $\Delta_{\mathcal{A}_1,6}^{\alpha_1}(n)$ for $2\leqslant n \leqslant10^5$}
   \end{minipage}\hfill
   \begin{minipage}{0.47\textwidth}
     \centering
     \includegraphics[width=.9\linewidth]{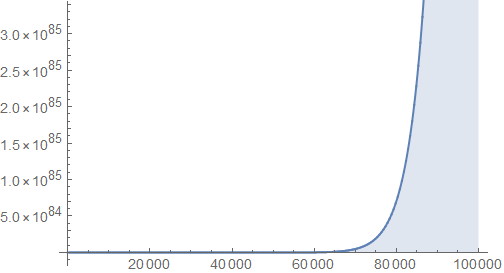}
     \caption{Values of $\Delta_{\mathcal{A}_1,7}^{\alpha_1}(n)$ for $2\leqslant n \leqslant10^5$}
   \end{minipage}
\end{figure}
\end{center}
These figures agree with Theorem \ref{thm2}. Moreover, for every $k\geqslant6$ one can explicitly determine when the inequality $\Delta_{\mathcal{A}_1,k}^{\alpha_1}(n)>0$ is true. Now, it is also interesting to illustrate the behavior of $\Delta_{\mathcal{A}_1,k}^{\alpha_2}(n)$, where $\alpha_2=5$ for some small numbers $k$.
\begin{center}
\begin{figure}[!htb]
   \begin{minipage}{0.47\textwidth}
     \centering
     \includegraphics[width=.9\linewidth]{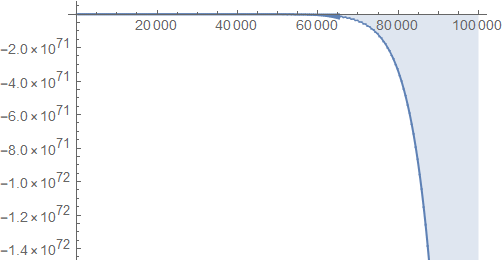}
     \caption{Values of $\Delta_{\mathcal{A}_1,5}^{\alpha_2}(n)$ for $2\leqslant n \leqslant10^5$}
   \end{minipage}\hfill
   \begin{minipage}{0.47\textwidth}
     \centering
     \includegraphics[width=.9\linewidth]{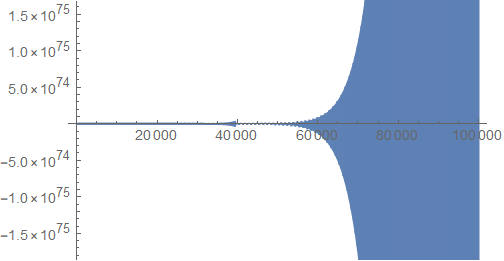}
     \caption{Values of $\Delta_{\mathcal{A}_1,6}^{\alpha_2}(n)$ for $2\leqslant n \leqslant10^5$}
   \end{minipage}
\end{figure}
\begin{figure}[!htb]
   \begin{minipage}{0.47\textwidth}
     \centering
     \includegraphics[width=.9\linewidth]{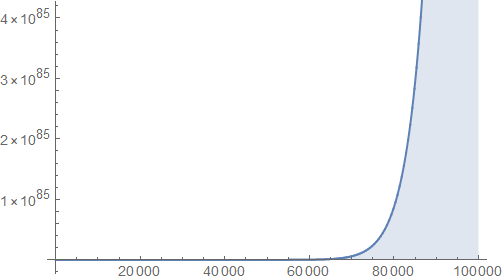}
     \caption{Values of $\Delta_{\mathcal{A}_1,7}^{\alpha_2}(n)$ for $2\leqslant n \leqslant10^5$}
   \end{minipage}\hfill
   \begin{minipage}{0.47\textwidth}
     \centering
     \includegraphics[width=.9\linewidth]{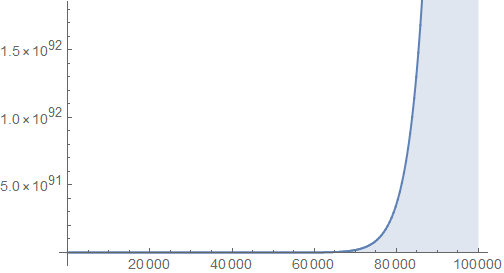}
     \caption{Values of $\Delta_{\mathcal{A}_1,8}^{\alpha_2}(n)$ for $2\leqslant n \leqslant10^5$}
   \end{minipage}
\end{figure}
\end{center}
Once again we observe that all the figures concur with our criterion. Furthermore, it follows from the proof of Theorem \ref{thm2} that there are infinitely many positive integers $n$ such that $\Delta_{\mathcal{A}_1,6}^{\alpha_2}(n)\leqslant0$.}
\end{ex}

\section{The \texorpdfstring{$r$}{TEXT}-log-concavity of quasi-polynomials like functions}

In this section, we explore a bit more complex problem --- the so-called $r$-log-concavity of $p_\mathcal{A}(n,k)$. Let us recall that a sequence of real numbers $\omega=\left(w_i\right)_{i=1}^\infty$ is said to be (asymptotically) $r$-log-concave for $r\in\mathbb{N}_+$, if for some $N$ all terms of the sequences
\begin{align*}
    \left(\widehat{\mathcal{L}}\omega\right)_i,\left(\widehat{\mathcal{L}}^2\omega\right)_i,\ldots,\left(\widehat{\mathcal{L}}^r\omega\right)_i
\end{align*}
are positive for every $i\geqslant N$, where
\begin{align*}
    \widehat{\mathcal{L}}\omega=\left(w_{i+1}^2-w_{i}w_{i+2}\right)_{i=0}^\infty \text{ and } \widehat{\mathcal{L}}^k\omega=\widehat{\mathcal{L}}\left(\widehat{\mathcal{L}}^{k-1}\omega\right)
\end{align*}
for $k\in\{2,3,\ldots,r\}$.
As it was mentioned in Introduction, the $1$-log-concavity of $(w_i)_{i=0}^\infty$ corresponds to the inequality $w_{n+1}^2-w_{n}w_{n+2}>0$; the $2$-log-concavity in addition demands that the following inequality
$$w_{n+2}^4-2w_{n+1}w_{n+2}^2w_{n+3}+w_{n+1}^2w_{n+2}w_{n+4}+w_{n}w_{n+2}w_{n+3}^2-w_{n}w_{n+2}^2w_{n+4}>0$$
is satisfied for all $n\geqslant N$. We resign from demonstrating other inequalities for $r\geqslant3$. 

The question arises whether we can successfully apply the methods from Section 2 to find out a convenient criterion for the $r$-log-concavity of the restricted partition function $p_\mathcal{A}(n,k)$. In fact, it is not the best idea to bound the left hand side of each of the inequalities from above and below. It is sill possible to do so for the $2$-log-concavity problem, but what about the $100$-log-concavity of $p_\mathcal{A}(n,k)$? Nota bene, if we take advantage of the aforementioned approach to resolve the $2$-log-concave case, then we immediately get the following. 

\begin{pr}\label{4.1}
    Let $\mathcal{A}=\left(a_i\right)_{i=1}^\infty$ and $k\geqslant7$ be fixed. If $\gcd A=1$ for every $(k-6)$-multisubset $A\subset\{a_1,a_2,\ldots,a_k\}$, then $\left(p_\mathcal{A}(n,k)\right)_{n=0}^\infty$ is asymptotically $2$-log-concave.
\end{pr}
\begin{proof}
    The assumptions from the statement together with Proposition \ref{Proposition4} assert that $p_\mathcal{A}(n,k)$ might be bounded from above and below by
    \begin{align*}
        f(n)<p_\mathcal{A}(n,k)<g(n),
    \end{align*}
    where $f(n)=c_{k-1}n^{k-1}+\cdots+c_{k-7}n^{k-7}-Fn^{k-8}$ and $g(n)=c_{k-1}n^{k-1}+\cdots+c_{k-7}n^{k-7}+Fn^{k-8}$, where all the coefficients $c_i$ are uniquely determined by Theorem \ref{2.4}, while $$F=\frac{\prod_{i=1}^k(1+iDk)}{k!\prod_{i=1}^ka_i}$$ with $D=\lcm(a_1,a_2,\ldots,a_k)$. Next, it is enough to examine the inequality of the form
    \begin{align*}
        f^4(n+2)&-2g(n+1)g^2(n+2)g(n+3)+f^2(n+1)f(n+2)f(n+4)\\&+f(n)f(n+2)f^2(n+3)-g(n)g^2(n+2)g(n+4)>0.
    \end{align*}
    After some tiresome calculations one can determine that the leading coefficient of the left hand side is equal to $588c_{k-1}$, where $$c_{k-1}=\frac{1}{(k-1)!\prod_{i=1}^ka_i}.$$
    Hence, the sequence $\left(p_\mathcal{A}(n,k)\right)_{n=0}^\infty$ must be $2$-log-concave for all but finitely many values of $n$, as required.
\end{proof}

Theorem \ref{log3} and Proposition \ref{4.1} can mislead us to believe that in order to receive the asymptotic $r$-log-concavity of $p_\mathcal{A}(n,k)$ (for some $r$), we need to require that $\gcd A=1$ for every $(k-(r+1)!)$-multisubset $A\subset\{a_1,a_2,\ldots,a_k\}$. But, it is not true as we see in the following.
\begin{ex}
   {\rm Let $\mathcal{A}_1=(1,2,3,4,5,6,\ldots)$ be a sequence of consecutive positive integers, as before. Numerical calculations made in Mathematica \cite{WM} illustrate how the sequence $\widehat{\mathcal{L}}_{\mathcal{A}_1,k}^2=\left(\widehat{\mathcal{L}}^2\left(p_{\mathcal{A}_1}(n,k)\right)_n\right)_{n=0}^\infty$ behaves for $7\leqslant k\leqslant10$ and $2\leqslant n\leqslant10^5$.
\begin{center}
\begin{figure}[!htb]
   \begin{minipage}{0.47\textwidth}
     \centering
     \includegraphics[width=.9\linewidth]{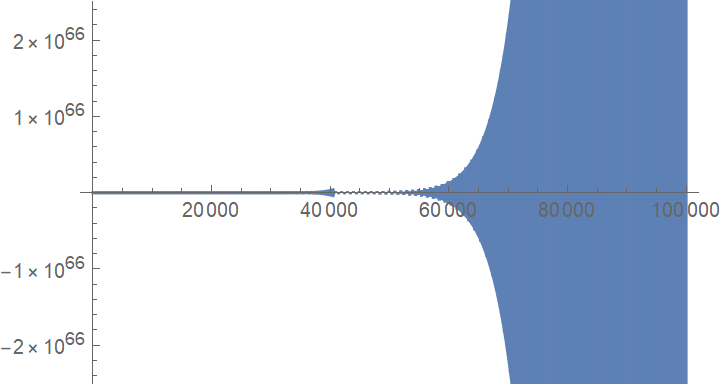}
     \caption{Values of $\left(\widehat{\mathcal{L}}_{\mathcal{A}_1,7}^2\right)_n$ for $2\leqslant n \leqslant10^5$}
   \end{minipage}\hfill
   \begin{minipage}{0.47\textwidth}
     \centering
     \includegraphics[width=.9\linewidth]{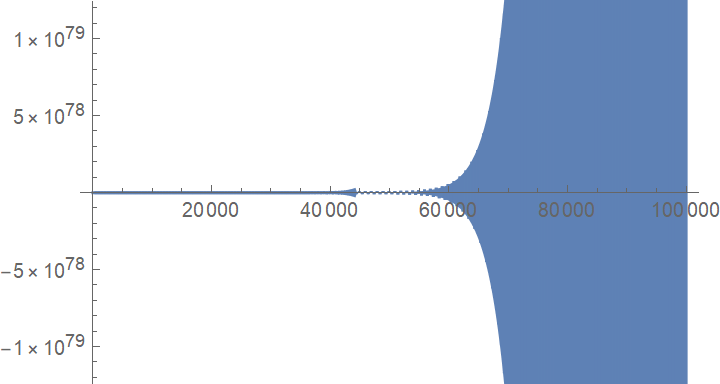}
     \caption{Values of $\left(\widehat{\mathcal{L}}_{\mathcal{A}_1,8}^2\right)_n$ for $2\leqslant n \leqslant10^5$}
   \end{minipage}
\end{figure}
\begin{figure}[!htb]
   \begin{minipage}{0.47\textwidth}
     \centering
     \includegraphics[width=.9\linewidth]{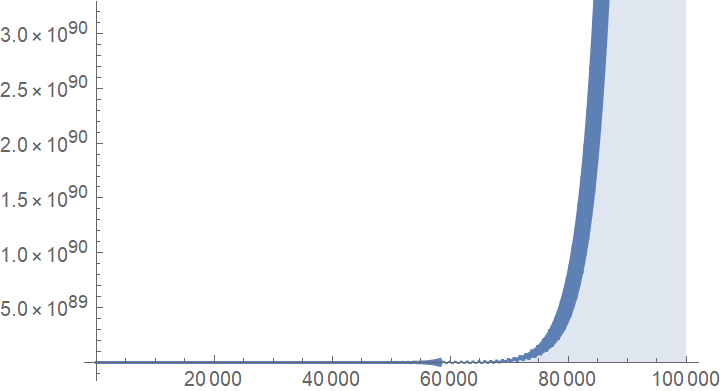}
     \caption{Values of $\left(\widehat{\mathcal{L}}_{\mathcal{A}_1,9}^2\right)_n$ for $2\leqslant n \leqslant10^5$}
   \end{minipage}\hfill
   \begin{minipage}{0.47\textwidth}
     \centering
     \includegraphics[width=.9\linewidth]{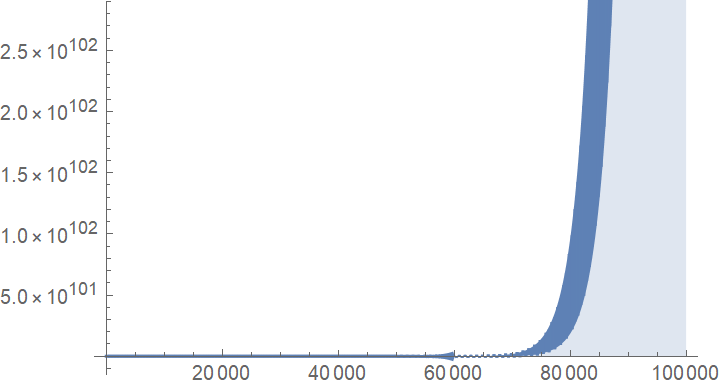}
     \caption{Values of $\left(\widehat{\mathcal{L}}_{\mathcal{A}_1,10}^2\right)_n$ for $2\leqslant n \leqslant10^5$}
   \end{minipage}
\end{figure}
\end{center}
By looking at both Figure 9 and Figure 10, we can not be completely certain that their corresponding restricted partitions functions are not $2$-log-concave for sufficiently large values of $n$. On the other hand, Figure 11 as well as Figure 12 suggest that the sequence $\widehat{\mathcal{L}}_{\mathcal{A}_1,k}^2(n)$ is (asymptotically) $2$-log-concave for all $k\geqslant9$, while Proposition \ref{4.1} only maintains that the property holds for all $k\geqslant13$ in this case.}
\end{ex}

Indeed, there is a bit more efficient criterion for the $r$-log concavity than Proposition \ref{4.1}. In order to find it, we apply a small modification of a very useful theorem obtained by Hou and Zhang \cite{Hou}. To state their result, we need a little preparation.

Let $\omega=\left(w_i\right)_{i=0}^\infty$ be a sequence of positive real numbers. We set
\begin{align*}
    \left(\mathcal{R}\omega\right)_i=\frac{w_{i+1}}{w_{i}}\hspace{0.3cm}\text{and}\hspace{0.2cm}\left(\mathcal{R}^2\omega\right)_i=\frac{\left(\mathcal{R}\omega\right)_{i+1}}{\left(\mathcal{R}\omega\right)_{i}}=\frac{w_iw_{i+2}}{w_{i+1}^2}.
\end{align*}
In particular, observe that the $1$-log-concavity of $\left(w_i\right)_{i=0}^\infty$ is equivalent to the statement $\left(\mathcal{R}^2w\right)_i<1$. Actually, it is also true for such a sequence of real numbers $\left(w_i\right)_{i=0}^\infty$ that both $\left(w_iw_{i+2}\right)_{i=0}^\infty$ and $\left(w_{i+1}^2\right)_{i=0}^\infty$ are positive.

Next, we assume that $\left(x_i\right)_{i=0}^\infty$ is a sequence of real numbers. Further, suppose that there are numbers $d_j$ and a $(m+1)$-tuple $\alpha=(\alpha_0,\alpha_1,\ldots,\alpha_m)$ such that
$$\alpha_0<\alpha_1<\cdots<\alpha_m$$
and
$$\lim_{n\to\infty}n^{\alpha_m}\left(x_n-\sum_{j=0}^m\frac{d_j}{n^{\alpha_j}}\right)=0.$$
The value
$$y_{n,\alpha}=\sum_{j=0}^m\frac{d_j}{n^{\alpha_j}}$$
is called the \textit{Puiseux-type} approximation of $x_n$ (of degree $\alpha_m$) and is denoted by $x_n\approx y_{n,\alpha}$. When it is clear from the context, we will write $y_n$ instead of $y_{n,\alpha}$. It turns out that only the first few summands of $y_{n}$ and $\alpha_m$ are important from our perspective. Therefore, we use the standard little-$o$ notation to write
$$x_n=y_{n}+o\left(\frac{1}{n^{\alpha_m}}\right)=\frac{d_0}{n^{\alpha_0}}+\frac{d_1}{n^{\alpha_1}}+\cdots+\frac{d_m}{n^{\alpha_m}}+o\left(\frac{1}{n^{\alpha_m}}\right).$$
Now, we are able to recall Hou and Zhang's criterion \cite{Hou}.

\begin{thm}{\label{HouZ}}
    Let $\omega=\left(w_i\right)_{i=0}^\infty$ be a sequence of positive real numbers such that $\left(\mathcal{R}^2\omega\right)_n$ has a Puiseux-type approximation $y_{n}$ of the form
    $$y_n=1+\frac{d_1}{n^{\alpha_1}}+\cdots+\frac{d_{m}}{n^{\alpha_m}},$$
    where $m\geqslant1$, $0<\alpha_1\leqslant\alpha_m$. If $d_1<0$ and $\alpha_1<2$, then $\left(w_i\right)_{i=0}^\infty$ is asymptotically $\floor{\alpha_m/\alpha_1}$-log-concave.
\end{thm}

For the sake of clarity, let us further write $p_{\mathcal{A},k}(n)=p_\mathcal{A}(n,k)$ whenever we apply an operator ($\mathcal{R}$ or $\widehat{\mathcal{L}}$) to $p_\mathcal{A}(n,k)$ in order to emphasize that we consider it as\linebreak a sequence in $n$. Unfortunately, we can not explicitly employ Theorem \ref{HouZ} to $p_\mathcal{A}(n,k)$, because for the restricted partition function we usually have that $\alpha_1=2$. More precisely, if $k\geqslant4$ and $\gcd A=1$ for all $(k-2)$-multisubsets $A$ of $\{a_1,a_2,\ldots,a_k\}$, then Theorem \ref{2.4} implies that 
$$p_\mathcal{A}(n,k)=an^{k-1}+bn^{k-2}+cn^{k-3}+c_{k-4}(n)n^{k-4}+\cdots+c_0(n),$$
where $c_{0}(n),\ldots,c_{k-4}(n)$ depend on the residue class of $n\pmod*{\lcm(a_1,a_2,\ldots,a_k)}$ and 
\begin{align*}
    a&=\frac{1}{(k-1)!\prod_{i=1}^ka_i},\\
    b&=\frac{\sigma}{2(k-2)!\prod_{i=1}^ka_i},\\
    c&=\frac{3\sigma^2-s_2}{24(k-3)!\prod_{i=1}^ka_i}
\end{align*}
with $\sigma=a_1+a_2+\cdots+a_k$ and $s_2=a_1^2+a_2^2+\cdots+a_k^2$. Now, it is straightforward to compute that
\begin{align*}
    p_\mathcal{A}(n,k)p_\mathcal{A}(n+2,k)=\beta_{2k-2}n^{2k-2}+\beta_{2k-3}n^{2k-3}+\beta_{2k-4}n^{2k-4}+\beta(n)
\end{align*}
and
\begin{align*}
    p_\mathcal{A}^2(n+1,k)=\gamma_{2k-2}n^{2k-2}+\gamma_{2k-3}n^{2k-3}+\gamma_{2k-4}n^{2k-4}+\gamma(n),
\end{align*}
where 
\begin{align*}
    \beta_{2k-2}&=\gamma_{2k-2}=a^2,\\
    \beta_{2k-3}&=\gamma_{2k-3}=2a(a(k-1)+b),\\
    \beta_{2k-4}&=\gamma_{2k-4}-a^2(k-1)=2a^2(k-1)(k-2)+2ab(2k-3)+2ac+b^2,
\end{align*}
and $\beta(n)$, together with $\gamma(n)$, is some quasi-polynomial of degree $2k-5$. Hence, the Puiseux-type approximation of 
$\left(\mathcal{R}^2p_{\mathcal{A},k}\right)(n)$ of degree $2$ might be calculated as follows:
\begin{align*}
    \left(\mathcal{R}^2p_{\mathcal{A},k}\right)(n)&=\frac{p_\mathcal{A}(n+2,k)p_\mathcal{A}(n,k)}{p_\mathcal{A}^2(n+1,k)}\\
    &=\frac{\beta_{2k-2}n^{2k-2}+\beta_{2k-3}n^{2k-3}+\beta_{2k-4}n^{2k-4}+\beta(n)}{\gamma_{2k-2}n^{2k-2}+\gamma_{2k-3}n^{2k-3}+\gamma_{2k-4}n^{2k-4}+\gamma(n)}\\
&=1+\frac{-a^2(k-1)n^{2k-4}+\beta(n)-\gamma(n)}{\gamma_{2k-2}n^{2k-2}+\gamma_{2k-3}n^{2k-3}+\gamma_{2k-4}n^{2k-4}+\gamma(n)}\\
&=1+\frac{1-k}{n^2}+o\left(\frac{1}{n^2}\right),
\end{align*}
where the last equality is a consequence of the following equality
\begin{align*}
    &\frac{-a^2(k-1)n^{2k-4}+\beta(n)-\gamma(n)}{\gamma_{2k-2}n^{2k-2}+\gamma_{2k-3}n^{2k-3}+\gamma_{2k-4}n^{2k-4}+\gamma(n)}-\frac{1-k}{n^2}\\=&\frac{n^2(\beta(n)-\gamma(n))+(k-1)(\gamma_{2k-3}n^{2k-3}+\gamma_{2k-4}n^{2k-4}+\gamma(n))}{n^2(\gamma_{2k-2}n^{2k-2}+\gamma_{2k-3}n^{2k-3}+\gamma_{2k-4}n^{2k-4}+\gamma(n))}.
\end{align*}
It should be pointed out that without any additional assumptions on the multiset $\{a_1,a_2,\ldots,a_k\}$, we can not simply present a bit more precise Puiseux-type approximation of $\left(\mathcal{R}^2p_{\mathcal{A},k}\right)(n)$ --- mainly because if we try to do so, then we obtain that the nominator over $n^3$ depends on the residue class of $n\pmod*{\lcm(a_1,a_2,\ldots,a_k)}$.

Let us now slightly modify Theorem \ref{HouZ} to get a relevant criterion in the case of the restricted partition function $p_\mathcal{A}(n,k)$. To achieve this goal we use two auxiliary lemmas. The first of them is due to Hou and Zhang \cite{Hou}.
\begin{lm}\label{HouZlemma}
Let $\omega=\left(w_i\right)_{i=0}^\infty$ be a sequence of real numbers. We have that 
\begin{align*}
\mathcal{R}^2\widehat{\mathcal{L}}\omega=\left(t_{n+1}^2\frac{(t_n-1)(t_{n+2}-1)}{(t_{n+1}-1)^2}\right)_{n=0}^\infty,
\end{align*}
where $\left(t_i\right)_{i=0}^\infty=\mathcal{R}^2\omega$.
\end{lm}
On the other hand, the second one gives us information about the Puiseux-type approximation of those functions which grow as fast as polynomials.

\begin{lm}\label{lemma2}
    Let $s\in\mathbb{N}_{\geqslant2}$ and $a_{l},a_{l-1},\ldots,a_{l-s}\in\mathbb{R}$ be arbitrary. Suppose further that there is a function $f(n)=a_ln^l+a_{l-1}n^{l-1}+\cdots+a_{l-s}n^{l-s}+o\left(n^{l-s}\right)$. Moreover, let us put $f(n)f(n+2)-f^2(n+1)=d_{2l-2,1}n^{2l-2}+\cdots+d_{2l-s,1}n^{2l-s}+o\left(n^{2l-s}\right)$ (in fact, $d_{2l,1}=d_{2l-1,1}=0$) and $f^2(n+1)=b_{2l}n^{2l}+\cdots+b_{2l-s}n^{2l-s}+o\left(n^{2l-s}\right)$. Then the equality
    \begin{align*}
        \left(\mathcal{R}^2f\right)(n)=1+\frac{d_2}{n^2}&+\cdots+\frac{d_j}{n^j}+\frac{\sum_{i=1}^{s-j}d_{2l-j-i,j}n^{2l-j-i}+o\left(n^{2l-s}\right)}{b_{2l}n^{2l}+\cdots+b_{2l-s}n^{2l-s}+o\left(n^{2l-s}\right)}
        \end{align*}
    holds for every $2\leqslant j\leqslant s$ with $d_{i}=d_{2l-i,i-1}/b_{2l}$, where 
    \begin{align*}
        d_{u,v}=d_{u,v-1}-d_{2l-v,v-1}b_{u+v}/b_{2l}
    \end{align*}
    for $2\leqslant v\leqslant s-1$ and $2l-s\leqslant u\leqslant2l-v-1$.
    In particular, the Puiseux-type approximation $y_{n}$ of $f(n)f(n+2)/f^2(n+1)$ takes the form
    \begin{align*}
        y_n=1+\frac{-l}{n^2}+\frac{d_{3}}{n^3}+\cdots+\frac{d_{s}}{n^{s}}+o\left(\frac{1}{n^{s}}\right).
    \end{align*}
\end{lm}
\begin{proof}
    Let $s\geqslant2$ and $f(n)$ be fixed. Employing the same method as in the discussion before Lemma \ref{HouZlemma}, we get that
    \begin{align}\label{eq1}
        \left(\mathcal{R}^2f\right)(n)=\frac{c_{2l}n^{2l}+c_{2l-1}n^{2l-1}+\cdots+c_{2l-s}n^{2l-s}+o\left(n^{2l-s}\right)}{b_{2l}n^{2l}+b_{2l-1}n^{2l-1}+\cdots+b_{2l-s}n^{2l-s}+o\left(n^{2l-s}\right)},
    \end{align}
    where 
    \begin{align*}
        c_{2l}=b_{2l}&=a_l^2\\
        c_{2l-1}=b_{2l-1}&=2a_l(la_l+a_{l-1})\\
        c_{2l-2}=b_{2l-2}-la_l^2&=2l(l-1)a_l^2+2(2l-1)a_la_{l-1}+2a_la_{l-2}+a_{l-1}^2,
    \end{align*}
    and all the remaining terms $b_i$ and $c_j$ can be explicitly determined. Hence, it is clear that
    \begin{align*}
        \left(\mathcal{R}^2f\right)(n)&=1+\frac{d_{2l-2,1}n^{2l-2}+\cdots+d_{2l-s,1}n^{2l-s}+o\left(n^{2l-s}\right)}{b_{2l}n^{2l}+b_{2l-1}n^{2l-1}+\cdots+b_{2l-2r}n^{2l-2r}+o\left(n^{2l-s}\right)}\\
        &=1+\frac{d_2}{n^2}+\frac{d_{2l-3,2}n^{2l-3}+\cdots+d_{2l-s,2}n^{2l-s}+o\left(n^{2l-s}\right)}{b_{2l}n^{2l}+b_{2l-1}n^{2l-1}+\cdots+b_{2l-2r}n^{2l-2r}+o\left(n^{2l-s}\right)}\\
        &=1+\frac{-l}{n^2}+o\left(\frac{1}{n^2}\right),
    \end{align*}
    as required. In particular, we proved the lemma for $j=2$. Now, we assume that for a fixed $s$ the claim holds for every $2\leqslant m \leqslant j-1$, and it suffices to verify its correctness for $m=j$. The induction hypothesis maintains that
    \begin{align*}
         \left(\mathcal{R}^2f\right)(n)=1+\frac{d_2}{n^2}&+\cdots+\frac{d_{j-1}}{n^{j-1}}+\frac{\sum_{i=1}^{s-j+1}d_{2l-j+1-i,j-1}n^{2l-j+1-i}+o\left(n^{2l-s}\right)}{b_{2l}n^{2l}+\cdots+b_{2l-s}n^{2l-s}+o\left(n^{2l-s}\right)}.
    \end{align*}
    Since $j\leqslant s$, we may also write
    \begin{align*}
        \left(\mathcal{R}^2f\right)(n)=1&+\frac{d_2}{n^2}+\cdots+\frac{d_{j-1}}{n^{j-1}}+\frac{d_{2l-j,j-1}/b_{2l}}{n^j}\\
        &+\frac{\sum_{i=j+1}^{s}\left(d_{2l-i,j-1}-d_{2l-j,j-1}b_{2l-i+j}/b_{2l}\right)n^{2l-i}+o\left(n^{2l-s}\right)}{b_{2l}n^{2l}+\cdots+b_{2l-s}n^{2l-s}+o\left(n^{2l-s}\right)}\\
        =1&+\frac{d_2}{n^2}+\cdots+\frac{d_{j-1}}{n^{j-1}}+\frac{d_{j}}{n^j}+\frac{\sum_{i=1}^{s-j}d_{2l-j-i,j}n^{2l-j-i}+o\left(n^{2l-s}\right)}{b_{2l}n^{2l}+\cdots+b_{2l-s}n^{2l-s}+o\left(n^{2l-s}\right)}.
    \end{align*}
    This completes the proof by the law of induction.
\end{proof}
The upcoming result is a key to resolve the $r$-log-concavity problem of both some functions `similar' to polynomials in general, and the restricted partition function in particular.

\begin{thm}\label{r-log-quasi}
     Let $l$ and $r$ be arbitrary positive integers such that $l\geqslant2r$. Suppose further that  $f(n)=a_ln^l+a_{l-1}n^{l-1}+\cdots+a_{l-2r}n^{l-2r}+o\left(n^{l-2r}\right)$. Then
     \begin{align*}
    \left(\mathcal{R}^{2}\widehat{\mathcal{L}}^{j-1}f\right)(n)=1+\frac{2^{j}-2^{j-1}l-2}{n^2}+q_j(n)+o\left(\frac{1}{n^{2(r-j+1)}}\right)
\end{align*}
for any $1\leqslant j\leqslant r$, where $q_j(n)$ is an expression of the form
$$q_j(n)=\frac{z_{3,j}}{n^3}+\cdots+\frac{z_{2(r-j+1),j}}{n^{2(r-j+1)}}$$
for some real numbers $z_{3,j},z_{4,j},\ldots,z_{2(r-j+1),j}$.
\end{thm}

\begin{proof}
    Let us fix parameters $l$, $r$ and a function $f(n)$. Applying Lemma \ref{lemma2} with $s=2r$ maintains that
\begin{align*}
        \left(\mathcal{R}^2f\right)(n)=\frac{f(n)f(n+2)}{f^2(n+1)}=1+\frac{-l}{n^2}+\frac{d_{3}}{n^3}+\cdots+\frac{d_{2r}}{n^{2r}}+o\left(\frac{1}{n^{2r}}\right)
    \end{align*}
with some real numbers $d_3,d_4,\ldots,d_{2r}$. This completes the proof in the case of $j=1$. Now, let us assume that the statement is valid for all $1\leqslant i \leqslant j-1$, and let us verify its correctness for $i=j$. It follows from the induction hypothesis that
\begin{align*}
    \left(\mathcal{R}^{2}\widehat{\mathcal{L}}^{j-2}f\right)(n)=1+\frac{2^{j-1}-2^{j-2}l-2}{n^2}+q_{j-1}(n)+o\left(\frac{1}{n^{2(r-j+2)}}\right),
\end{align*}
where $q_{j-1}(n)$ is an expression of the form
\begin{align*}
    q_j(n)=\frac{z_{3,j-1}}{n^3}+\cdots+\frac{z_{2(r-j+2),j-1}}{n^{2(r-j+2)}}
\end{align*}
for some real numbers $z_{3,j-1},z_{4,j-1},\ldots,z_{2(r-j+2),j-1}$. To simplify the notation put 
\begin{align*}
    t_n=1+\frac{m}{n^2}+q_{j-1}(n)+o\left(\frac{1}{n^{2(r-j+2)}}\right)
\end{align*}
with $m=2^{j-1}-2^{j-2}l-2$. Then, we have
\begin{align*}
    \left(\mathcal{R}^{2}\widehat{\mathcal{L}}^{j-2}f\right)(n)=t_n.
\end{align*}
 Our goal is to determine Puiseux-type approximation of $\left(\mathcal{R}^2\widehat{\mathcal{L}}^{j-1}f\right)(n)$ (of degree $2(r-j+1)$). From the generalized binomial theorem, we get that 
    \begin{align*}
        \frac{1}{(n+u)^v}&=\sum_{i=0}^\infty\binom{-v}{i}n^{-v-i}u^i=\sum_{i=0}^\infty(-1)^i\binom{v+i-1}{i}\frac{u^i}{n^{v+i}}\\
        &=\sum_{i=0}^{2(r-j+2)-v}(-1)^i\binom{v+i-1}{i}\frac{u^i}{n^{v+i}}+o\left(\frac{1}{n^{2(r-j+2)}}\right)\numberthis \label{eqn1}
    \end{align*}
    holds for all positive integers $u$ and $v$ such that $v\leqslant2(r-j+2)$. Therefore, it is straightforward to see that
    \begin{align*}
        t_{n+1}^2&=\left(1+\frac{m}{(n+1)^2}+q_{j-1}(n+1)+o\left(\frac{1}{(n+1)^{2(r-j+2)}}\right)\right)^2\\
        &=1+\frac{2m}{n^2}+\widetilde{q}_{j-1}(n)+o\left(\frac{1}{n^{2(r-j+2)}}\right)\numberthis \label{eqn2},
    \end{align*}
    where 
    \begin{align*}
        \widetilde{q}_{j-1}(n)=\frac{\widetilde{z}_{3,j-1}}{n^3}+\cdots+\frac{\widetilde{z}_{2(r-j+2),j-1}}{n^{2(r-j+2)}}
\end{align*}
for some real numbers $\widetilde{z}_{3,j-1},\widetilde{z}_{4,j-1},\ldots,\widetilde{z}_{2(r-j+2),j-1}$. For the sake of brevity, we put
    \begin{align}\label{eqn3}
        w_n=\frac{n^2}{m}(t_n-1).
    \end{align}
    Thus, it is clear that 
    \begin{align*}
        w_n&=1+\frac{n^2}{m}\cdot q_{j-1}(n)+o\left(\frac{1}{n^{2(r-j+1)}}\right)\\
        &=1+\frac{z_{3,j-1}/m}{n}+\cdots+\frac{z_{2(r-j+2),j-1}/m}{n^{2(r-j+1)}}+o\left(\frac{1}{n^{2(r-j+1)}}\right)\\
        &=1+\frac{\widehat{z}_{1}}{n}+\cdots+\frac{\widehat{z}_{2(r-j+1)}}{n^{2(r-j+1)}}+o\left(\frac{1}{n^{2(r-j+1)}}\right),
    \end{align*}
    where $\widehat{z}_i=z_{i+2,j-1}/m$ for $1\leqslant i \leqslant 2(r-j+1)$. We apply Taylor expansion of the logarithm function in order to estimate $w_nw_{n+2}/w_{n+1}^2$. It follows that
    \begin{align*}
        \log{w_n}&=\sum_{i=1}^\infty\frac{(-1)^{i+1}}{i}(w_n-1)^i\\
        &=\frac{e_1}{n}+\frac{e_2}{n^2}+\cdots+\frac{e_{2(r-j+1)}}{n^{2(r-j+1)}}+o\left(\frac{1}{n^{2(r-j+1)}}\right),
    \end{align*}
    with $e_1=\widehat{z}_1$, $e_2=\widehat{z}_2-\widehat{z}_1^2/2$ and all the remaining values $e_3,e_4,\ldots,e_{2(r-j+1)}$ are uniquely determined in terms of $\widehat{z}_1,\widehat{z}_2,\ldots,\widehat{z}_{2(r-j+1)}$. After some tiresome but elementary calculations, we obtain that
    \begin{align*}
        \log{\frac{w_nw_{n+2}}{w_{n+1}^2}}&=\log{w_n}+\log{w_{n+2}}-2\log{w_{n+1}}\\
&=\frac{e_1}{n}+\cdots+\frac{e_{2(r-j+1)}}{n^{2(r-j+1)}}+\frac{e_1}{n+2}+\cdots+\frac{e_{2(r-j+1)}}{(n+2)^{2(r-j+1)}}\\
&\phantom{=}-2\left(\frac{e_1}{n+1}+\cdots+\frac{e_{2(r-j+1)}}{(n+1)^{2(r-j+1)}}\right)+o\left(\frac{1}{n^{2(r-j+1)}}\right)\\
&=\frac{\widehat{e}_3}{n^3}+\cdots+\frac{\widehat{e}_{2(r-j+1)}}{n^{2(r-j+1)}}+o\left(\frac{1}{n^{2(r-j+1)}}\right),
    \end{align*}
    where $\widehat{e}_3=2e_1$, $\widehat{e}_4=6(e_2-e_1)$, and each of the numbers $\widehat{e}_5,\widehat{e}_6,\ldots,\widehat{e}_{2(r-j+1)}$ is uniquely determined in terms of $e_1,e_2,\ldots,e_{2(r-j+1)}$. In particular, it means that 
\begin{align}\label{eqn4}
        \frac{w_nw_{n+2}}{w_{n+1}^2}=1+\frac{\check{e}_3}{n^3}+\frac{\check{e}_4}{n^4}+\cdots+\frac{\check{e}_{2(r-j+1)}}{n^{2(r-j+1)}}+o\left(\frac{1}{n^{2(r-j+1)}}\right),
    \end{align}
where all the numbers $\check{e}_3,\check{e}_4,\ldots,\check{e}_{2(r-j+1)}$ can be explicitly presented in terms of $\widehat{e}_3,\widehat{e}_4,\ldots,\widehat{e}_{2(r-j+1)}$, for example, we have $\check{e}_3=\widehat{e}_3$ and $\check{e}_4=\widehat{e}_4$. Hence, by (\ref{eqn1}), (\ref{eqn3}) and (\ref{eqn4}) we have that
    \begin{align*}
        \frac{(t_n-1)(t_{n+2}-1)}{(t_{n+1}-1)^2}&=\frac{(n+1)^4}{n^2(n+2)^2}\cdot\frac{w_nw_{n+2}}{w_{n+1}^2}\\
        &=\frac{(n+1)^4}{n^2}\cdot\left(\sum_{i=0}^{2(r-j+1)}(-1)^i\binom{i+1}{i}\frac{2^i}{n^{i+2}}+o\left(\frac{1}{n^{2(r-j+2)}}\right)\right)\\
        &\phantom{=\frac{(n+1)^4}{n^2}}\times\left(1+\frac{\check{e}_3}{n^3}+\cdots+\frac{\check{e}_{2(r-j+1)}}{n^{2(r-j+1)}}+o\left(\frac{1}{n^{2(r-j+1)}}\right)\right)\\
        &=1+\frac{2}{n^2}+\frac{\widetilde{e}_3}{n^3}+\cdots+\frac{\widetilde{e}_{2(r-j+1)}}{n^{2(r-j+1)}}+o\left(\frac{1}{n^{2(r-j+1)}}\right),
    \end{align*}
    where all the values $\widetilde{e}_3,\widetilde{e}_4,\ldots,\widetilde{e}_{2(r-j+1)}$ can be systematically derived, e.g. $\widetilde{e}_3=\check{e}_3-4$. Next, Lemma \ref{HouZlemma} and the equality (\ref{eqn2}) assert that
    \begin{align*}
        \left(\mathcal{R}^2\left(\widehat{\mathcal{L}}^{j-1}p_{\mathcal{A},k}\right)\right)(n)&=\left(1+\frac{2m}{n^2}+\widetilde{q}_{j-1}(n)+o\left(\frac{1}{n^{2(r-j+2)}}\right)\right)\\
        &\times\left(1+\frac{2}{n^2}+\frac{\widetilde{e}_3}{n^3}+\cdots+\frac{\widetilde{e}_{2(r-j+1)}}{n^{2(r-j+1)}}+o\left(\frac{1}{n^{2(r-j+1)}}\right)\right)\\
        &=1+\frac{2m+2}{n^2}+q_{j}(n)+o\left(\frac{1}{n^{2(r-j+1)}}\right)\\
        &=1+\frac{2^j-2^{j-1}l-2}{n^2}+q_{j}(n)+o\left(\frac{1}{n^{2(r-j+1)}}\right),
    \end{align*}
    where 
    $$q_j(n)=\frac{z_{3,j}}{n^3}+\cdots+\frac{z_{2(r-j+1),j}}{n^{2(r-j+1)}}$$
for some real numbers $z_{3,j},z_{4,j},\ldots,z_{2(r-j+1),j}$. This completes the inductive step and thereby ends the proof.
\end{proof}

We are ready to present the main result of the paper. Surprisingly, we prove that in two different ways.

\begin{thm}\label{r-log-2}
    Let $l$ and $r$ be arbitrary positive integers such that $l\geqslant2r$. If $f(n)=a_ln^l+a_{l-1}n^{l-1}+\cdots+a_{l-2r}n^{l-2r}+o\left(n^{l-2r}\right)$, then the sequence $\left(f(n)\right)_{n=0}^\infty$ is asymptotically $r$-log-concave.
\end{thm}

\begin{proof}
    For fixed parameters $l$ and $r$ and a function $f(n)$, Theorem \ref{r-log-quasi} points out that
\begin{align*}
    \left(\mathcal{R}^{2}\widehat{\mathcal{L}}^{j-1}f\right)(n)=1+\frac{2^{j}-2^{j-1}l-2}{n^2}+q_j(n)+o\left(\frac{1}{n^{2(r-j+1)}}\right)
\end{align*}
for any $1\leqslant j\leqslant r$, where $q_j(n)$ is an expression of the form
$$q_j(n)=\frac{z_{3,j}}{n^3}+\cdots+\frac{z_{2(r-j+1),j}}{n^{2(r-j+1)}}$$
for some real numbers $z_{3,j},z_{4,j},\ldots,z_{2(r-j+1),j}$.
Since we have that $l\geqslant2r$, one can easily deduce that the inequalities

\begin{align*}
    2^{j}-2^{j-1}l-2\leqslant2^{j}-2^{j}r-2<0
\end{align*}
    are true for every $1\leqslant j\leqslant r$. But in particular it means that the sequence $\left(f(n)\right)_{n=0}^\infty$ is asymptotically $j$-log-concave for any such a $j$, as required.
\end{proof}

Let us also present an alternative, more direct, approach to the problem.

\begin{proof*}{Second proof of Theorem~\ref{r-log-2}}
    Under the assumptions from the statement, we prove by induction that
\begin{align}\label{eqn5}\left(\widehat{\mathcal{L}}^sf\right)(n)&=a_l^{2^s}\left(\prod_{j=1}^sm_j^{2^{s-j}}\right)n^{m_{s+1}}+q_s(n)+o\left(n^{m_{s+1}-2(r-s)}\right),
\end{align}
holds for each $1\leqslant s\leqslant r$, where $m_i=2^{i-1}l+2-2^i$ and
\begin{align*}
q_s(n)=\delta_{m_{s+1}-1}n^{m_{s+1}-1}+\cdots+\delta_{m_{s+1}-2(r-s)}n^{m_{s+1}-2(r-s)}
\end{align*}
with some real numbers $\delta_i$. For $s=1$, we have that
\begin{align*} \left(\widehat{\mathcal{L}}f\right)(n)&=f^2(n+1)-f(n)f(n+2)\\
&=a_l^2ln^{2l-2}+d_{2l-3}n^{2l-3}+\cdots+d_{2l-2r}n^{2l-2r}+o\left(n^{2l-2r}\right),
\end{align*}
where $d_{2l-2},d_{2l-3},\ldots,d_{2l-2r}$ are some real numbers, as required. Therefore, let us assume that the claim is true for every $1\leqslant i \leqslant s-1<r$, and verify its correctness for $i=s$. It follows from the induction hypothesis that

\begin{align*}
\left(\widehat{\mathcal{L}}^{s-1}f\right)(n)&=a_l^{2^{(s-1)}}\left(\prod_{j=1}^{s-1}m_j^{2^{s-1-j}}\right)n^{m_{s}}+q_{s-1}(n)+o\left(n^{m_{s}-2(r-s+1)}\right),
\end{align*}
where $q_{s-1}$ is of the required form. Now, we can just write
\begin{align*}
\left(\widehat{\mathcal{L}}^{s}f\right)(n)&=\left(\widehat{\mathcal{L}}^{s-1}f\right)^2(n+1)-\left(\widehat{\mathcal{L}}^{s-1}f\right)(n)\left(\widehat{\mathcal{L}}^{s-1}f\right)(n+2)\\
&=\left(a_l^{2^{(s-1)}}\left(\prod_{j=1}^{s-1}m_j^{2^{s-1-j}}\right)\right)^2m_sn^{2m_{s}-2}+q_s(n)+o\left(n^{2m_{s}-2(r-s)-2}\right)\\
&=a^{2^s}\left(\prod_{j=1}^{s}m_j^{2^{s-j}}\right)n^{m_{s+1}}+q_s(n)+o\left(n^{m_{s+1}-2(r-s)}\right),
\end{align*}
where
\begin{align*}
q_s(n)=\delta_{m_{s+1}-1}n^{m_{s+1}-1}+\cdots+\delta_{m_{s+1}-2(r-s)}n^{m_{s+1}-2(r-s)}
\end{align*}
with some real numbers $\delta_i$. This ends the proof of (\ref{eqn5}) by the law of induction. 

Now, it is enough to see that the leading coefficient of (\ref{eqn5}) is positive for any $1\leqslant s\leqslant r$. Thus the sequence $\left(f(n)\right)_{n=0}^\infty$ is asymptotically $r$-log-concave, as required.
\end{proof*}

It is worth noting that despite the fact that the former proof requires more sophisticated preparation than the latter one, the method which we use there might be also effectively applied for more complicated functions like the partition function \cite{Hou2} or the overpartition function \cite{Mu}. 

In particular, Theorem \ref{r-log-2} delivers us the following efficient criterion for the $r$-log-concavity of the restricted partition function.

\begin{thm}\label{rlog}
    Let $\mathcal{A}=\left(a_i\right)_{i=1}^\infty$, $r\in\mathbb{N}_+$ and $k>2r$ be fixed. Suppose further that $\gcd A=1$ for all $(k-2r)$-multisubsets $A$ of $\{a_1,a_2,\ldots,a_k\}$. Then the sequence $\left(p_\mathcal{A}(n,k)\right)_{n=0}^\infty$ is asymptotically $r$-log-concave.
\end{thm}
\begin{proof}
    The property is a direct consequence of Theorem \ref{2.4} and Theorem \ref{r-log-2}.
\end{proof}
There appears a natural question whether Theorem \ref{rlog} is optimal in a sense that it can not be further generalized for any other sequences $\mathcal{A}$. In fact, it is true, which directly follows from the following more general result.

\begin{thm}\label{r-log-3}
    Let $l$ and $r$ be arbitrary positive integers such that $l\geqslant2r$. Suppose further that we have 
    \begin{align*}
        f(n)=a_l(n)n^l+a_{l-1}(n)n^{l-1}+\cdots+a_{l-2r}(n)n^{l-2r}+o\left(n^{l-2r}\right),
    \end{align*} 
    where the coefficients $a_{l-2r}(n),\ldots,a_l(n)$ might depend on the residue class of\linebreak $n\pmod*{M}$ for some positive integer $M\geqslant2$. Then the sequence $\left(f(n)\right)_{n=0}^\infty$ is asymptotically $r$-log-concave if and only if all the numbers $a_{l-2r}(n),\ldots,a_l(n)$ are independent of the residue class of $n\pmod*{M}$.
\end{thm}
\begin{proof}
    The implication to the left hand side is clear by Theorem \ref{r-log-2}. To deal with the implication to the right hand side let us present the reasoning by induction on $r$. First, let us assume that $r=1$ and suppose, for contradiction, that at least one of the coefficients $a_{l}(n)$, $a_{l-1}(n)$ or $a_{l-2}(n)$ depend on the residue class of $n\pmod*{M}$. Let us set $t=\max\{j\in\{l-2,l-1,l\}:a_j(n)\text{ is not constant}\}$. It is enough to consider such residue class of $n\pmod*{M}$ that $a_t(n+1)$ is the smallest and at least one of $a_t(n)$ or $a_t(n+2)$ is strictly bigger than $a_t(n+1)$, and compute $\widehat{\mathcal{L}}f(n)$. In that case, we get that the leading coefficient is negative. Hence, the sequence $\left(f(n)\right)_{n=0}^\infty$ can not be asymptotically $1$-log-concave, as required.
    
    Now, let us assume that the statement is valid for each $1\leqslant\widetilde{r}<r$ and let us examine whether it is also true for $\widetilde{r}=r$. The induction hypothesis asserts that
    
    \begin{align*}
        f(n)=\sum_{i=l-2r+2}^la_in^i+a_{l-2r+1}(n)n^{l-2r+1}+a_{l-2r}(n)n^{l-2r}+o\left(n^{l-2r}\right),
    \end{align*}
    otherwise $f(n)$ could not be even asymptotically $(r-1)$-log-concave. Suppose for contradiction that at least one of the coefficients $a_{l-2r}(n)$ or $a_{l-2r+1}(n)$ depend on the residue class of $n\pmod*{M}$. Let us assume that $a_{l-2r+1}(n)$ is not independent of the residue class of $n\pmod*{M}$ --- the alternative reasoning for $a_{l-2r}(n)$ is analogous, and we leave it as an exercise for the reader. The coefficient $b_{2l-2r+1}(n)$ of
    \begin{align*}
        \widehat{\mathcal{L}}f(n)=\left(\sum_{i=2l-2r+2}^{2l-2}b_in^i\right)+b_{2l-2r+1}(n)n^{2l-2r+1}+b_{2l-2r}(n)n^{2l-2r}+o\left(n^{2l-2r}\right)
    \end{align*}
    is of the form 
    \begin{align*}
        b_{2l-2r+1}(n)=\left(2a_{l-2r+1}(n+1)-a_{l-2r+1}(n)-a_{l-2r+1}(n+2)\right)a_l-q,
    \end{align*}
    where $q$ is a constant depending on $a_{l-2r+2},\ldots,a_{l}$. Therefore, if we take such\linebreak a residue class of $n\pmod*{M}$ that $a_{l-2r+1}(n+1)$ is the smallest and at least one of $a_{l-2r+1}(n)$ or $a_{l-2r+1}(n+2)$ is strictly bigger than $a_{l-2r+1}(n+1)$, and we do the opposite (considering the largest possible coefficient $a_{l-2r+1}(n+1)$ and at least one of $a_{l-2r+1}(n)$ or $a_{l-2r+1}(n+2)$ strictly smaller than $a_{l-2r+1}(n+1)$), then we deduce that the coefficient $b_{2l-2r+1}(n)$ of $\widehat{\mathcal{L}}f(n)$ depends on the residue class of $n\pmod*{M}$. Thus the induction hypothesis points out that $\widehat{\mathcal{L}}f(n)$ can not be asymptotically $(r-1)$-log-concave which implies that $f(n)$ is not $r$-log-concave, as required.
\end{proof}

\begin{thm}\label{rlog2}
    Let $\mathcal{A}=\left(a_i\right)_{i=1}^\infty$, $r\in\mathbb{N}_+$ and $k>2r$ be fixed. Then the sequence $\left(p_\mathcal{A}(n,k)\right)_{n=0}^\infty$ is asymptotically $r$-log-concave if and only if we have that $\gcd A=1$ for all $(k-2r)$-multisubsets $A$ of $\{a_1,a_2,\ldots,a_k\}$.
\end{thm}
\begin{proof}
    The implication to the left hand side follows from Theorem \ref{rlog}. On the other hand, the implication to the right hand side is a direct consequence of both Theorem \ref{r-log-3} and Proposition \ref{pr5.7}.
\end{proof}

It is worth noting that one can use the above criterion to explicitly calculate when the $r$-log-concavity of $p_\mathcal{A}(n,k)$ holds for given parameters $\mathcal{A}$, $k$ and $r$. 

Until now, we have not discussed the easiest case of the quasi-polynomial-like function. Therefore, for the sake of completeness, let us present a criterion for the $r$-log-concavity of a polynomial.
\begin{cor}\label{cor4.11}
    Let $f(n)=a_kn^k+a_{k-1}n^{k-1}+\cdots+a_0\in\R[n]$ be a polynomial of $\deg(f)=k$. Then the sequence $\left(f(n)\right)_{n=0}^\infty$ is
    \begin{enumerate}
        \item not asymptotically $1$-log-concave if and only if $k=0$;
        \item at most asymptotically $1$-log-concave if and only if $k=1$;
        \item asymptotically $r$-log-concave for any $r\geqslant1$ if and only if $k\geqslant2$.
    \end{enumerate}
\end{cor}
\begin{proof}
    The statement can be easily verified for $k\leqslant1$. If $k\geqslant2$, then Theorem \ref{r-log-2} asserts that the sequence $\left(f(n)\right)_{n=0}^\infty$ is asymptotically $1$-log-concave. Moreover, we know that
    \begin{equation*}
        \widehat{\mathcal{L}}f(n)=b_{2k-2}n^{2k-2}+b_{2k-3}n^{2k-3}+\cdots+b_0
    \end{equation*}
    for some real numbers $b_0,b_1,\ldots,b_{2k-2}$ with $b_{2k-2}>0$. Since $2k-2\geqslant2$, we obtain that the sequence $\left(\widehat{\mathcal{L}}f(n)\right)_{n=0}^\infty$ is asymptotically $1$-log-concave (in particular, $\left(f(n)\right)_{n=0}^\infty$ is asymptotically $2$-log-concave). By repeating the above procedure, we deduce the required property. 
\end{proof}

\begin{cor}
    If $\mathcal{A}=(1,\textcolor{blue}{1},\textcolor{red}{1},\ldots)$, then the sequence $\left(p_\mathcal{A}(n,k)\right)_{n=0}^\infty$ is
    \begin{enumerate}
        \item not asymptotically $1$-log-concave if and only if $k=1$;
        \item at most asymptotically $1$-log-concave if and only if $k=2$;
        \item asymptotically $r$-log-concave for any $r\geqslant1$ if and only if $k\geqslant3$.
    \end{enumerate}
\end{cor}
\begin{proof}
    That is a direct consequence of the formula (\ref{formula11}) and Corollary \ref{cor4.11}.
\end{proof}

At the end of the manuscript, let us illustrate how Theorem \ref{rlog2} works in practice.

\begin{ex}
{\rm Let  $\mathcal{A}=(1,2,\textcolor{blue}{2},3,\textcolor{blue}{3},\textcolor{red}{3},\ldots)$ be the sequence of consecutive positive integers such that every number $j$ appears in $j$ distinct colors. For instance, if $k=10$, then $p_\mathcal{A}(n,k)$ takes the form
\begin{align*}
    p_\mathcal{A}(n,10)=\frac{n^9}{10032906240}+\frac{n^8}{74317824}+\frac{65 n^7}{83607552}+\frac{25 n^6}{995328}+q_{n\pmod*{12}}(n),
\end{align*}
where $q_{n\pmod*{12}}(n)$ is a quasi-polynomial part of $p_\mathcal{A}(n,10)$ depending on the residue class of $n\pmod*{12}$. In fact, one can check that
{\small \begin{align*}
    q_i(n)=\begin{cases}\frac{14863 n^5}{29859840}+\frac{1555 n^4}{248832}+\frac{533 n^3}{10752}+\frac{2909 n^2}{12096}+\frac{1703 n}{2520}+1, & \text{if}\ i=0,\\
        \frac{118661 n^5}{238878720}+\frac{98305 n^4}{15925248}+\frac{8143 n^3}{172032}+\frac{17403629 n^2}{83607552}+\frac{1475950039 n}{3344302080}+\frac{84987001}{286654464}, & \text{if}\ i=1,\\
        \frac{14863 n^5}{29859840}+\frac{1555 n^4}{248832}+\frac{533 n^3}{10752}+\frac{621449 n^2}{2612736}+\frac{3838811 n}{6531840}+\frac{401951}{1119744}, & \text{if}\ i=2,\\
        \frac{118661 n^5}{238878720}+\frac{98305 n^4}{15925248}+\frac{8171 n^3}{172032}+\frac{665135 n^2}{3096576}+\frac{21523559 n}{41287680}+\frac{75979}{131072}, & \text{if}\ i=3,\\
        \frac{14863 n^5}{29859840}+\frac{1555 n^4}{248832}+\frac{533 n^3}{10752}+\frac{78767 n^2}{326592}+\frac{140323 n}{204120}+\frac{2171}{2187}, & \text{if}\ i=4,\\
        \frac{118661 n^5}{238878720}+\frac{98305 n^4}{15925248}+\frac{8143 n^3}{172032}+\frac{17288941 n^2}{83607552}+\frac{1338324439 n}{3344302080}+\frac{12504185}{286654464}, & \text{if}\ i=5,\\
        \frac{14863 n^5}{29859840}+\frac{1555 n^4}{248832}+\frac{533 n^3}{10752}+\frac{23083 n^2}{96768}+\frac{49771 n}{80640}+\frac{317}{512}, & \text{if}\ i=6,\\
        \frac{118661 n^5}{238878720}+\frac{98305 n^4}{15925248}+\frac{8171 n^3}{172032}+\frac{18015989 n^2}{83607552}+\frac{1782402199 n}{3344302080}+\frac{164068921}{286654464}, & \text{if}\ i=7,\\
        \frac{14863 n^5}{29859840}+\frac{1555 n^4}{248832}+\frac{533 n^3}{10752}+\frac{78319 n^2}{326592}+\frac{131923 n}{204120}+\frac{1618}{2187}, & \text{if}\ i=8,\\
        \frac{118661 n^5}{238878720}+\frac{98305 n^4}{15925248}+\frac{8143 n^3}{172032}+\frac{642455 n^2}{3096576}+\frac{17740199 n}{41287680}+\frac{39819}{131072}, & \text{if}\ i=9,\\
        \frac{14863 n^5}{29859840}+\frac{1555 n^4}{248832}+\frac{533 n^3}{10752}+\frac{625033 n^2}{2612736}+\frac{4107611 n}{6531840}+\frac{685087}{1119744}, & \text{if}\ i=10,\\
        \frac{118661 n^5}{238878720}+\frac{98305 n^4}{15925248}+\frac{8171 n^3}{172032}+\frac{17901301 n^2}{83607552}+\frac{1644776599 n}{3344302080}+\frac{91586105}{286654464}, & \text{if}\ i=11.
    \end{cases}
\end{align*}}

Now, if we take $n\equiv1\pmod*{12}$, then
\begin{align*}
\widehat{\mathcal{L}}^2p_{\mathcal{A},10}(n)=-\frac{283n^{30}}{3909057129171792215334771260129280000}+q(n),
\end{align*}
where $q(n)$ is some quasi-polynomial of degree $29$. Therefore, the sequence \linebreak $\left(p_\mathcal{A}(n,10)\right)_{n=0}^\infty$ can not be asymptotically $2$-log-concave. On the other hand, if we consider $p_\mathcal{A}(n,11)$ and make similar computations to those ones above, then we get that the sequence  $\left(p_\mathcal{A}(n,11)\right)_{n=0}^\infty$ is $2$-log-concave for all $n\geqslant11320$, but is not asymptotically $3$-log-concave. It is worth noting that in this case we need to consider $60$ quasi-polynomials instead of $12$. Moreover, Theorem \ref{rlog2} asserts that the sequence $\left(p_\mathcal{A}(n,k)\right)_{n=0}^\infty$ is asymptotically $2$-log-concave for every $k\geqslant11$. Further, we can also repeat the aforementioned approach and deduce that the sequence $\left(p_\mathcal{A}(n,12)\right)_{n=0}^\infty$ is not asymptotically $3$-log-concave. However, if we investigate $\left(p_\mathcal{A}(n,12)\right)_{n=0}^\infty$ in that regard, we obtain that it is $3$-log-concave 
for each $n\geqslant607475$. Once again, Theorem \ref{rlog2} points out that the sequence $\left(p_\mathcal{A}(n,k)\right)_{n=0}^\infty$ is asymptotically $3$-log-concave for every $k\geqslant13$.}
\end{ex}

\section*{Acknowledgements}
I would like to thank Piotr Miska and Maciej Ulas for their time and valuable suggestions. This research was funded by both a grant of the National Science Centre (NCN), Poland, no. UMO-2019/34/E/ST1/00094 and a grant from the Faculty of Mathematics and Computer Science under the Strategic Program Excellence Initiative at the Jagiellonian University in Kraków.

\end{document}